\documentclass{cpc}

\usepackage{amsmath}
\usepackage{amsfonts}
\usepackage{amssymb}
\usepackage{subfigure}
\usepackage{graphicx}
\usepackage{color}

\usepackage{pinlabel}

\newtheorem{evaluation}{Evaluation}

  \DeclareMathOperator{\cross}{cr}
\DeclareMathOperator{\white}{wh}
\DeclareMathOperator{\black}{bl}
\DeclareMathOperator{\total}{tot}

\DeclareMathOperator{\ws}{ws}
\DeclareMathOperator{\cs}{cs}
\DeclareMathOperator{\bs}{bs}
\newcommand{\la}{\lambda}

\title[Evaluations of topological Tutte polynomials]{Evaluations of topological Tutte polynomials}

  \title[Evaluations of topological Tutte polynomials]{Evaluations of topological Tutte polynomials}
\author[J.~Ellis-Monaghan and I. Moffatt]{J.\ns E\ls L\ls L\ls I\ls S\ls -\ls M\ls O\ls N\ls A\ls G\ls H\ls A\ls N{$^1$}%
\thanks{Supported by the National Science Foundation (NSF) under grant number DMS-1001408 and by the Vermont Space Grant Consortium through the National Aeronautics and Space Administration (NASA)}\ns
and\ns I.\ns  M\ls O\ls F\ls F\ls A\ls T\ls T{$^2$} }
\affil{{$^1$} Department of Mathematics, Saint Michael's College, 1 Winooski Park, Colchester, VT 05439, USA.  \\
 (e-mail: \texttt{jellis-monaghan@smcvt.edu})\\[6pt]
{$^2$} Department of Mathematics,
Royal Holloway,
University of London,
Egham,
Surrey,
TW20 0EX, UK. \\
 (e-mail: \texttt{iain.moffatt@rhul.ac.uk})}

\recdate{XXXXX}
\begin{document}

\setcounter{page}{0}
\thispagestyle{empty}
\tableofcontents

\label{firstpage}
\maketitle%[P]

\begin{abstract}
We find new properties of the topological transition polynomial of embedded graphs, $Q(G)$.  We use these properties to   explain the striking similarities between certain evaluations of Bollob\'as and Riordan's ribbon graph polynomial, $R(G)$, and the topological Penrose polynomial, $P(G)$.
The general framework provided by $Q(G)$ also leads to several other combinatorial interpretations these polynomials.  In particular, we express $P(G)$, $R(G)$, and the Tutte polynomial, $T(G)$, as sums of chromatic polynomials of graphs derived from $G$; show that these polynomials count  $k$-valuations of medial graphs; show that $R(G)$ counts edge 3-colourings; and reformulate the Four Colour Theorem in terms of $R(G)$. We conclude with a reduction formula for the transition polynomial of the tensor product of two embedded graphs, showing that it leads to additional relations among these polynomials and to  further combinatorial interpretations of $P(G)$ and $R(G)$.
\end{abstract}

\begin{AMScodes} {05C31}{05C10}\end{AMScodes}

\section{Introduction}
In this paper we give versatile new identities for the topological transition polynomial $Q(G)$.  We use these identities  to obtain new connections between, and interpretations of, the Penrose polynomial  $P(G)$, and  Bollob\'as and Riordan's ribbon graph polynomial $R(G)$, and  also to explain similarities between known combinatorial evaluations of $P(G)$ and $R(G)$. (Definitions of these polynomials can be found in  Sections~\ref{sstgp} and~\ref{ss.tra}.)

We are motivated by some strikingly similar combinatorial interpretations of $P(G)$ and $R(G)$  in terms of edge colourings of medial graphs. If $G$ is an embedded graph (i.e., a graph in a surface as in Section~\ref{s.not}), a   $k$-valuation  of  its medial graph,  $G_m$, is an edge $k$-colouring of $G_m$ such that each vertex is incident to an even number (possibly zero) of edges of each colour.  A $k$-valuation of a checkerboard-coloured embedded medial graph yields four possible arrangements of colours about a vertex, which we term white, black, crossing, or total,  as in Figure~\ref{f.kval}. We let $\total(\phi)$ and $\cross(\phi)$ denote the number of  total vertices and crossing vertices, respectively, in  a $k$-valuation $\phi$.

\begin{figure}[h]
\begin{tabular}{ccccccc}
\labellist \small\hair 2pt
\pinlabel {$i$}  at 22 64
\pinlabel {$j$}  at 22 10
\pinlabel {$j$}  at 50 10
\pinlabel {$i$}  at 50 64
\endlabellist
 \includegraphics[width=1.2cm]{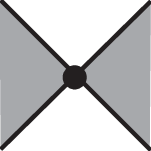}
  &  \quad\quad\quad &
  \labellist \small\hair 2pt
\pinlabel {$i$}  at 22 64
\pinlabel {$i$}  at 22 8
\pinlabel {$j$}  at 50 8
\pinlabel {$j$}  at 50 64
\endlabellist
 \includegraphics[width=1.2cm]{m1}
  & \quad\quad\quad&
  \labellist \small\hair 2pt
\pinlabel {$i$}  at 22 64
\pinlabel {$j$}  at 22 10
\pinlabel {$i$}  at 50 10
\pinlabel {$j$}  at 50 64
\endlabellist
 \includegraphics[width=1.2cm]{m1}
  & \quad\quad\quad&
  \labellist \small\hair 2pt
\pinlabel {$i$}  at 22 64
\pinlabel {$i$}  at 22 10
\pinlabel {$i$}  at 50 10
\pinlabel {$i$}  at 50 64
\endlabellist
 \includegraphics[width=1.2cm]{m1},
  \\
white && black &&crossing &&  total
\end{tabular}
\caption{Classifying $k$-valuations at a vertex, where $i\neq j$.}
\label{f.kval}
\end{figure}

The polynomials $R(G)$  and $P(G)$ admit the following remarkably similar  interpretations as sums over $k$-valuations of $G_m$:
\begin{align}\label{e.intro1}
 k^{c(G)}  R(G;k+1,k,1/k,1) &= \sum 2^{\total(\phi)},
\\
\label{e.intro2}
P(G;k) &= \sum  (-1)^{\cross(\phi)},
\\
\label{e.intro3}
P(G;-k) &= (-1)^{f(G)}  \sum  2^{\total(\phi)}.
\end{align}
The sum in \eqref{e.intro1} is over all $k$-valuations $\phi$ of $G_m$ that contain no crossing configurations, the sum in  \eqref{e.intro2} is over those that contain no black or crossing configurations, and the sum in \eqref{e.intro3} is over those that contain no black configurations.  In addition, in \eqref{e.intro3}, $G$ must be orientable and checkerboard colourable.
(Equation~\eqref{e.intro1} is due to Korn and Pak \cite{KP03}; while \eqref{e.intro2} and \eqref{e.intro3} are from   \cite{EMMb}, with \eqref{e.intro2} extending results of  Jaeger \cite{Ja90} and Aigner \cite{Ai97}.)

The striking fact that  $R(G)$ and $P(G)$ have such similar combinatorial evaluations demands explanation, and we find it here through the topological transition polynomial, $Q(G)$.  We unify and extend  Equations~\eqref{e.intro1} and \eqref{e.intro2} by finding a combinatorial interpretation of $Q(G)$ in Theorem~\ref{t.qk}.  Using the facts that  $Q(G)$ specialises to $P (G)$, and to $R(G)$ along certain curves, we show that \eqref{e.intro1} and \eqref{e.intro2}  arise as specialisations of  a combinatorial evaluation of $Q(G)$ (see Evaluation~\ref{c.qk}), thus explaining their similarity.

The connection between  \eqref{e.intro1} and \eqref{e.intro3} is especially interesting as it arises from  duality identities. By considering the twisted duality relation for the transition polynomial from \cite{EMMa}, in  Theorem~\ref{PtoZ} we connect $P(G^*)$ and $R(G^{\times})$, where $G^*$ is the geometric dual of $G$ and  $G^{\times}$ is its Petrie dual.
This new relation, together with our combinatorial evaluation of  $Q(G)$, allow us to extend Equation~\eqref{e.intro3} to a larger class of graphs (see Evaluation~\ref{total2}). This amalgamates Equations~\eqref{e.intro1}--\eqref{e.intro3} into a single result about the transition polynomial,  explaining not only the similarities between $P(G)$ and $R(G)$, but also the restriction on $G$ in Equation~\eqref{e.intro3}.

The new properties of $Q(G)$ also lead to a host of other interpretations of, and relations between,   graph polynomials.  The point of this is two-fold: it deepens the understanding of the relations between various graph polynomials; and, perhaps more importantly, it reveals  combinatorial and topological information that is encoded by evaluations of these polynomials.  For example, we  show in Theorem~\ref{c41.n5} that the Penrose polynomial can be written as a sum of chromatic polynomials:
\[  P(G;\lambda) = \sum_{A\subseteq E(G)}  (-1)^{ |A|}  \chi ((   G^{\tau(A)}   )^*   ;\lambda)  .\]
This identity completes a result of Aigner, from \cite{Ai97}, in which it was shown that $\chi(G^*;k)\leq P(G;k)$, for all $k\in \mathbb{N}$, and extends the version in \cite{EMMb} from plane graphs to all embedded graphs. We give a similar identity for  the ribbon graph and Tutte polynomials (see Corollary~\ref{t.interp2a}).
  We also give a number of new combinatorial evaluations of, and identities for, the ribbon graph polynomial $R(G)$ including a  reformulation of the Four Colour Theorem (Section~\ref{ss.argp}).

Many of the results mentioned so far rely on an identity connecting $P(G^*)$ and $R(G^{\times})$, which  arose from the  twisted duality relation for the transition polynomial from \cite{EMMa}. In Section~\ref{s.tp} we present a family of connections that is obtained instead by considering tensor products of embedded graphs. If $G$ and $H$ are abstract or embedded graphs, then,  roughly speaking, the tensor product of $G\otimes H$ is obtained by replacing  each edge of $G$ with a copy of $H-e$. In \cite{Bry82}, Brylawski showed how the Tutte polynomial  of $G\otimes H$ can be written in terms of the Tutte polynomial of its tensor factors. Brylawski's result was extended   to $R(G)$ in \cite{HM}. In Theorem~\ref{t.tensor}, our third central result, we give a tensor product formula for the transition polynomial, expressing $Q(G\otimes H)$ in terms of $Q(G)$. This leads to a new family of  identities relating the $P(G)$ and $R(G)$, and to new interpretations of these polynomials, and to connections with some results of Jaeger from \cite{Ja88}.

\section{Notation for embedded graphs}\label{s.not}
In this section, we briefly review relevant constructions involving embedded graphs. More leisurely expositions can be found in \cite{Ch1}, in Sections~2.1--2.3 of \cite{EMMb}, and Section~2.1 of \cite{Mof11c}.  Our notation is standard.

\subsection{Embedded graphs}
Cellularly embedded graphs, ribbon graphs, and arrow presentations all represent the same object, which we call an \emph{embedded graph}. We will move freely between these three descriptions, which we presently describe.

A {\em cellularly embedded  graph} $G=(V(G),E(G)) \subset \Sigma$ is a graph embedded in a surface $\Sigma$ such that each connected component of $\Sigma \backslash G$
is homeomorphic to a disc. The connected components, when viewed as subsets of $\Sigma$, are  the {\em faces}  of $G$.  Two cellularly embedded graphs $G\subset \Sigma$ and  $G'\subset \Sigma'$  are
 {\em equivalent},   if there is an (orientation preserving when $\Sigma$ is orientable) homeomorphism $\varphi:\Sigma \rightarrow \Sigma'$ with the property that $\varphi|_{G}:G\rightarrow G'$ is a graph isomorphism. We consider cellularly embedded graphs up to this equivalence.

A {\em ribbon graph} $G =\left(  V(G),E(G)  \right)$ is a surface with boundary represented as the union of two  sets of  discs, a set $V (G)$ of {\em vertices}, and a set $E (G)$ of {\em edges}  such that:  the vertices and edges intersect in disjoint line segments,
 each such line segment lies on the boundary of precisely one
vertex and precisely one edge, and
 every edge contains exactly two such line segments.
Ribbon graphs are known to be equivalent to cellularly embedded graphs (ribbon graphs arise as neighbourhoods of cellularly embedded graphs, while capping off the the boundary components of a ribbon graph with discs results in a cellularly embedded graph).
We say that two  ribbon graphs are {\em equivalent} if they define equivalent cellularly embedded graphs. Ribbon graphs are considered up to this equivalence.

An {\em arrow presentation} consists of a set of circles, each marked with a collection of disjoint, labelled arrows, called  {\em marking arrows}  on them such that there are exactly two marking arrows with each label.  Arrow presentations describe  ribbon graphs as follows: if $G$ is a ribbon graph, for each edge $e$, arbitrarily orient its boundary, place an $e$-labelled arrow pointing in the direction of the orientation on the arcs where  $e$ intersects a vertex. Taking the boundaries of the vertices together with the labelled arrows gives an arrow presentation.   In the other direction,  given an arrow presentation, obtain a ribbon graph by  identifying each circle with the boundary of a vertex disk.    Then, for each label $e$, take a disc with an orientation of its boundary and  identify an arc on its boundary with an $e$ labelled arrow such that the direction of each arrow agrees with the orientation. Thus, in an arrow presentation, each circle corresponds to a vertex of an embedded graph, and each pair of $e$-labelled arrows corresponds to an edge.
 We say that two arrow presentations are {\em equivalent} if they describe the same ribbon graph (and therefore the same cellularly embedded graph). Arrow presentations are considered up to  equivalence.

 Let $G$ be an embedded graph. We set  $v(G):=|V(G)|$ and  $e(G):=|E(G)|$, and denote the number of components of $G$ by $c(G)$. The rank of $G$ is $r(G):=v(G)-c(G)$, and the nullity of $G$ is $n(G):=e(G)-r(G)$.
  If $G$ is viewed as a ribbon graph, then $f(G)$ is the number of boundary components of the surface defining the ribbon graph, and if $G$ is viewed as a cellularly embedded graph, $f(G)$ is the number of its faces.   The {\em Euler genus}, $\gamma(G) $, of  $G$ equals the genus of $G$ if $G$ is non-orientable, and is twice its genus if $G$ is orientable. Euler's formula gives that $v(G)-e(G)+f(G)=2c(G)-\gamma(G)$.
 An embedded graph $G$ is   {\em plane}  if it is connected and $\gamma(G)=0$.

We define {\em subgraphs} of embedded graphs in terms of ribbon graphs. A ribbon graph $H $ is a {\em ribbon subgraph}  of $G$ if $H$ can be obtained by deleting vertices and edges of $G$. If $H$ is a ribbon subgraph of $G$ with $V(H)=V(G)$, then $H$ is a  {\em spanning ribbon subgraph} of $G$.
We define embedded subgraphs for other realisations of embedded graphs by translating into the language of ribbon graphs.  In particular, this means that a cellularly embedded subgraph $H$ of a cellularly embedded graph $G$ is automatically equipped with a cellular embedding in a surface, although  $G$ and $H$ may be embedded in different surfaces.
(For example, let  $G$ be the orientable ribbon graph consisting of one vertex $v$ and two edges $e$ and $f$ that meet $v$ in the cyclic order $(e\,f\,e\,f)$, and let $H=G-e$. When viewed as a cellularly embedded graph, $G$ is embedded on the torus with one loop following a meridian and the other a longitude. However, as a cellularly embedded graph, $H$ is a loop embedded in the sphere.)
 For $A\subseteq E(G)$, we let $r(A)$, $c(A)$, $n(A)$,  $f(A)$,  $g(A)$, and $\gamma(A)$ each refer to the spanning subgraph of $(V(G),A)$ of $G$ (where $G$ is given by context).

\subsection{Deletion and contraction}\label{ss.contract}
Care has to be taken when deleting and contracting edges of embedded graphs to ensure that the operation results in an embedded graph.

 If $e\in E(G)$ then $G-e$ is the spanning subgraph on edge set $E(G)\backslash \{e\}$.  When viewed in the setting of ribbon graphs, then it is clear that $G-e$ is again a ribbon graph, i.e. is cellularly embedded, although not necessarily embedded in the same surface as $G$ as noted in the torus example above.

 We will define contraction through arrow presentations. If $G$ is an embedded graph viewed as an arrow presentation and $e\in E(G)$, then $G/e$ is obtained as follows. Suppose $\alpha$ and $\beta$ are the two $e$-labelled arrows. Connect the tip of $\alpha$ to the tail of $\beta$ with a line segment; and  connect the tip of $\beta$ to the tail of $\alpha$ with another line segment. Delete $\alpha$, $\beta$ and the arcs of the circles (or circle) on which they lie. This results in an arrow presentation for $G/e$. (See Table~\ref{tablecontractrg}.)
This definition of contraction agrees with the usual idea of contracting an edge by identifying it and its incident vertices in the cases where doing so results in a   cellularly embedded graph or  a ribbon graphs.
 Table~\ref{tablecontractrg} shows the local effect of contracting an edge of a ribbon graph. In it, the ribbon graphs are identical outside of the region shown.

\begin{table}
\centering
\begin{tabular}{|c||c|c|c|c|}\hline
& any edge &  non-loop & non-orientable loop&orientable loop\\ \hline
\raisebox{6mm}{$G$} &
\includegraphics[scale=.45]{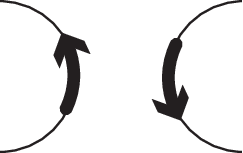}
&\includegraphics[scale=.25]{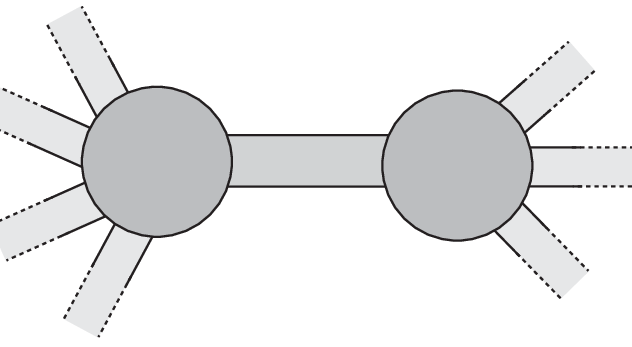} &\includegraphics[scale=.25]{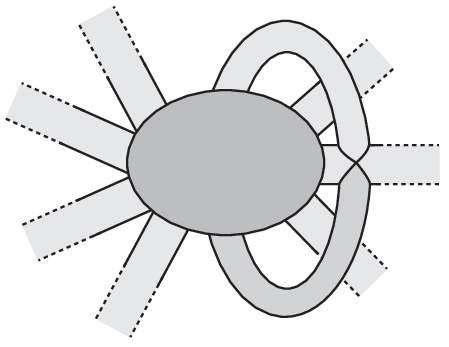} &\includegraphics[scale=.25]{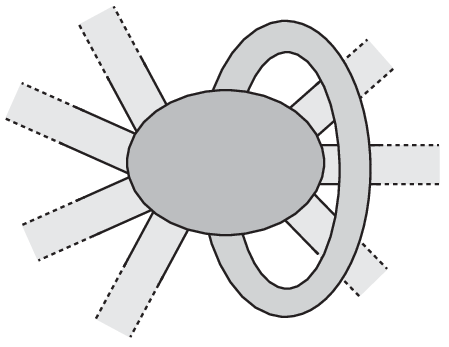}
\\ \hline
\raisebox{6mm}{$G/e$} &
\includegraphics[scale=.45]{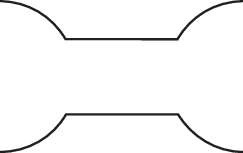}
&\includegraphics[scale=.25]{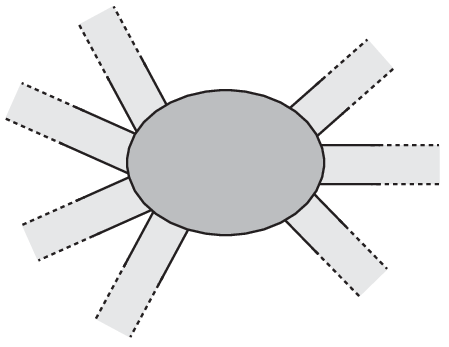} &\includegraphics[scale=.25]{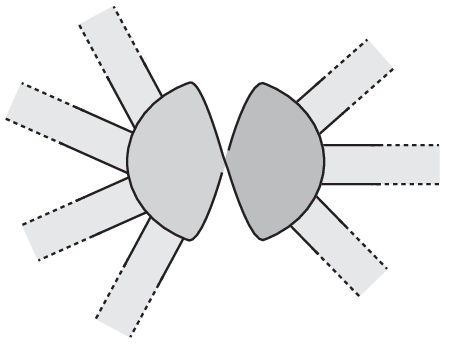}&\includegraphics[scale=.25]{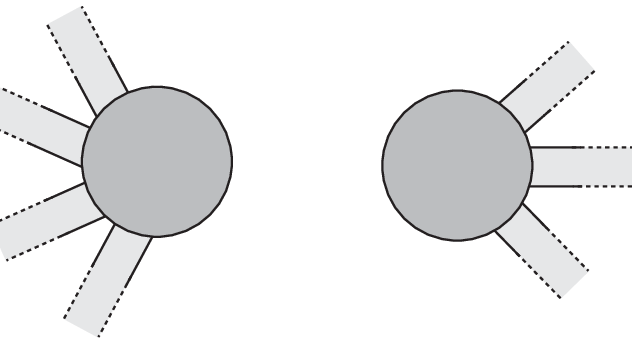} \\ \hline
\end{tabular}
\caption{Contracting  an edge of an arrow presentation and a ribbon graph}
\label{tablecontractrg}
\end{table}

\subsection{Petrie duality and Geometric duality}\label{ss.dual}
Let $G$ be an embedded graph viewed as a cellularly embedded graph. Recall that the {\em geometric dual}, $G^*$, of  $G$ is formed  by
placing one vertex in each face of $G$ and embedding an edge of $G^*$
between two vertices whenever the faces of $G$ they lie in are
adjacent. Note that $G$ and $G^*$ are embedded in the same surface.

 The {\em Petrie dual}, or {\em Petrial},  $G^{\times}$, of  $G$ is the embedded graph  with  the same edges and vertices as $G$, but with Petrie polygons, which are the result of closed left-right walks in $G$, as the faces (see Wilson \cite{Wil79}).  Petrie duals have a very simple definition  in terms of arrow presentations in that reversing the direction of exactly one arrow of each label in an arrow presentation of $G$ results in the arrow presentation of $G^{\times}$.
Observe that in the language of arrow presentations, Petrie duality is a local operation rather than a global operation, and thus we can form the Petrie dual with respect to only a subset of edges. Let $G$ be an embedded graph represented by an arrow presentation and $A\subseteq E(G)$. Then the {\em partial Petrial} $G^{\tau(A)}$ is the embedded graph obtained by, for each $e\in A$, reversing the direction of exactly one of the  $e$-labelled arrows. Note that $G^{\times}=G^{\tau(E(G))}$ and that $G^{\emptyset}=G$. If $A=\{e\}$ then we will write $G^{\tau(e)}$ for $G^{\tau(\{e\})}$.
  In terms of ribbon graphs and working informally, $G^{\times}$  is obtained by detaching one end of each edge from its incident vertex disc, giving the edge a half-twist, and reattaching it to the vertex disc. The partial petrial $G^{\tau(A)}$ of a ribbon graph $G$ is formed by adding a half-twist only to the edges in $A$. Note that, by equivalence, any number of full-twists may be added or removed from an edge.

\subsection{Medial graphs} \label{Medial graphs}
Let $G$ be a cellularly embedded graph. Its {\em medial graph}, $G_m$, is the embedded graph constructed by  placing a vertex on each edge of $G$, and then drawing the edges of the medial graph by following the face boundaries of $G$ (so  each vertex of $G_m$ is of degree $4$). The medial graph of an isolated vertex is a closed loop. (Medial graphs may thus have {\em free loops}, where a free loop is a circular edge with no incident vertex.) Note that if $G$ is cellularly embedded in a surface $\Sigma$, then $G_m$ is also cellularly embedded in $\Sigma$.  The vertices of $G_m$ correspond to the edges of $G$, and we will generally index $V(G_m)$ using $E(G)$: if $e$ is an edge of $G$, we denote the corresponding vertex of $G_m$ by $v_e$.

A {\em checkerboard colouring} of  a cellularly embedded graph is an assignment of the colour black or white to each face such that adjacent faces receive different colours (i.e., it is a face 2-colouring).
If $G_m$ is the medial graph of $G$, then $G_m$ can be  checkerboard coloured by  colouring the faces corresponding to a vertex of $G$ black, and the remaining faces white.  We call this checkerboard colouring the \emph{canonical checkerboard colouring}. Note that $G$ must be specified to determine the canonical checkerboard colouring of $G_m$, and that medial graphs are always checkerboard colourable.

\section{Graph polynomials and $k$-valuations}\label{s.kval}
\subsection{Topological graph polynomials}\label{sstgp}
We start by briefly reviewing the definitions of the ribbon graph polynomial and the Penrose polynomial.

\subsubsection{The ribbon graph polynomial}

In \cite{BR1} and \cite{BR2}, Bollob\'as and Riordan extended the  Tutte polynomial from graphs to embedded graphs, resulting in the  ribbon graph polynomial, $R(G)$.

   \begin{definition}[Bollob\'as and Riordan \cite{BR2}]\label{defBR}
Let $G$ be an embedded graph.  Then the \emph{ribbon graph polynomial} or  {\em Bollob\'as-Riordan polynomial}, $R(G;x,y,z) \in \mathbb{Z}[x,y,z]$,  is defined by
\[   R(G;x,y,z) = \sum_{A \subseteq E( G)}   (x - 1)^{r( G ) - r( A )}   y^{n(A)} z^{c(A) - f(A) + n(A)}  . \]
  \end{definition}
Although  $R(G)$ often appears with a fourth variable that records the orientability of each embedded spanning subgraph, here we omit the fourth variable as it plays no role in our results.

The   Tutte polynomial, $T(G)$, is a specialisation of $R(G)$:
$ T(G;x,y)=R(G;x,y-1,1)$.  In addition,  as the exponent of $z$ is equal to the Euler genus $\gamma(A)$, we have  that $T(G)$ and $R(G)$ agree when $G$ is a plane graph:
 \begin{equation}\label{TR1b}
T(G;x,y)=R(G;x,y-1, z).
\end{equation}

\subsubsection{The Penrose polynomial}
The Penrose polynomial   was defined implicitly by
Penrose in~\cite{Pen71} for plane graphs and was extended to all embedded graphs  in \cite{EMMb}. This extension to embedded graphs reveals new properties of the Penrose polynomial  that can not be realised exclusively in terms of plane graphs.
The (plane) Penrose polynomial has been defined in terms of bicycle spaces, left-right facial walks, or states of a medial graph. Here, however, we use a formulation of the polynomial from  \cite{EMMb} to define it in terms of partial Petrials.
\begin{definition}\label{d.Penrose}
Let  $G$ be an embedded graph. Then the {\em Penrose polynomial}, $P(G;\la)\in \mathbb{Z}[\la]$ is
\[P(G;\la) := \sum_{A\subseteq E(G)}   (-1)^{|A|} \lambda^{f(G^{\tau(A)})}.\]
\end{definition}

The Penrose polynomial encodes a lot of information about graph colouring.
For example, it is well-known that proving the Four Colour Theorem is equivalent
to proving that every plane, cubic, connected graph can be properly
edge-coloured with three colours.   The Penrose polynomial encodes exactly this information
(see \cite{Pen71}): if $G$ is a plane, cubic, connected graph, then
\begin{equation}\label{e.ptcol}
\text{the number of edge 3-colourings of } G = P\left( {G;3} \right) = \left(  - 1/4
\right)^{\frac{{v(G)}}{2}} P\left( {G; - 2} \right) .
\end{equation}
For further explorations of its properties we refer the reader to the excellent  expositions on the classical Penrose polynomial given by Aigner in~\cite{Ai97} and~\cite{Aig00}, and  to \cite{EMMb} for its extension to embedded graphs.

\subsection{$k$-valuations}\label{ss.kva}

\begin{definition}\label{c4.d.kval}
Let $G$ be an embedded graph and $G_m$ be its canonically checkerboard coloured medial  graph. A {\em $k$-valuation}  of $G_m$ is an edge $k$-colouring $\phi: E(G_m)\rightarrow \{1, 2, \ldots , k\}$ such that for each $i$ and every vertex $v_e$ of $G_m$, the number of $i$-coloured edges incident to $v_e$ is even. We let $\mathcal{K}(G_m,k)$ denote the set of  $k$-valuations of $G_m$.
\end{definition}

Suppose $G_m$ is canonically checkerboard coloured. Then in a $k$-valuation $\phi$, at each vertex $v_e\in V(G_m)$, one number appears four times, or two distinct numbers appear twice. There are exactly four ways that this can happen, and we use the canonical checkerboard colouring to classify the four types $( \phi , v_e )$ as in Figure~\ref{f.kval}.

A $k$-valuation $\phi$ of $G_m$ is:
\begin{itemize}
\item {\em admissible} if $( \phi , v_e)$ is  a white or crossing type for each $v_e\in V(G_m)$, and we let $\mathcal{A}(G_m,k)$ denote the set of admissible $k$-valuations;
 \item {\em $R$-permissible} if $( \phi , v_e)$ is a white, black, or total type for each $v_e\in V(G_m)$, and we let $\mathcal{R}(G_m,k)$ denote the set of $R$-permissible $k$-valuations;
\item {\em $P$-permissible} if $( \phi , v_e)$ is a white, crossing, or total type for each $v_e\in V(G_m)$, and we let $\mathcal{P}(G_m,k)$ denote the set of $P$-permissible $k$-valuations.
\end{itemize}
If $\phi$ is a $k$-valuation, we let $\white(\phi)$, $\black(\phi)$, $\cross(\phi)$, and $\total(\phi)$,  denote, respectively, the total  number of white, black, crossing, and total types in $\phi$.

As mentioned in the introduction, $R(G)$ and $P(G)$ both count $k$-valuations.
Korn and Pak found the following combinatorial evaluation of the ribbon graph polynomial.
 \begin{evaluation}[Korn and Pak \cite{KP03}] \label{t.KP}
Let $G$ be an embedded graph, and  $k \in \mathbb{N}$. Then
\begin{equation}\label{e.KP}
k^{c(G)}  R(G;k+1,k,1/k) = \sum_{  \phi \in \mathcal{R}(G_m,k)  }   2^{\total(\phi)}.
\end{equation}
\end{evaluation}
Turning to the Penrose polynomial, Jaeger proved that when $G$ is a plane graph, $P(G;k)$ can be expressed as a count of $k$-valuations.
\begin{evaluation}[Jaeger \cite{Ja90}]\label{th.addval}
Let $G$ be a plane  graph, and $k \in \mathbb{N}$. Then $P(G;k) =  |\mathcal{A}(G_m,k)| $.
\end{evaluation}
The authors extended Jaeger's interpretation to all embedded graphs by showing that $P(G;k)$ gives a signed count of admissible $k$-valuations:
\begin{evaluation}[\cite{EMMb}]\label{sumcross}  Let $G$ be an embedded graph, and $k \in \mathbb{N}$. Then
\begin{equation}\label{e.sumcross}
P(G;k) = \sum_{  \phi \in \mathcal{A}(G_m,k)  }   (-1)^{\cross(\phi)}.
\end{equation}
\end{evaluation}

We note that Evaluation~\ref{th.addval} can be recovered from Evaluation~\ref{sumcross} since, for plane graphs, $\cross(\phi) \equiv 0\mod 2$ (see \cite{EMMb} for details).

For certain graphs, the Penrose polynomial at negative integers, $P(G;-k)$, can  also be described by a count of $k$-valuations:
\begin{evaluation}[\cite{EMMb}]\label{t.total}  If $G$ is orientable and checkerboard colourable, then
\begin{equation}\label{e.total}
P(G;-k) = (-1)^{f(G)}  \sum_{  \phi \in \mathcal{P}(G_m,k)  }   2^{\total(\phi)}.
\end{equation}
\end{evaluation}

Evaluation \ref{t.total} does not hold for general embedded graphs (for example, consider the 1-path). However, in Evaluation \ref{total2}, we will see that Equation \eqref{e.total} holds for a somewhat larger class of embedded graphs than conditioned here.

\section{Unifying the combinatorial interpretations}\label{s.unif}

\subsection{The transition polynomial} \label{ss.tra}
The generalised transition polynomial, $q(G; W,t)$,  of \cite{E-MS02} is a multivariate graph polynomial that generalises Jaeger's transition polynomial of~\cite{Ja90}.  In \cite{EMMa}, the authors specialised the generalised transition polynomial to embedded graphs, calling this specialisation the \emph{topological transition polynomial}. Although both the generalised and the topological transition polynomials encode vertex-specific weight systems, as in \cite{EMMa} and \cite{E-MS02}, it suffices for our purposes here to use weight systems which are the same at each vertex.  However, several of the results we present may readily be extended to the more general setting should the need arise.

A vertex state at a vertex $v$  of a  $4$-regular graph arises from a partition, into pairs, of the half-edges incident to $v$. If $(u,v)$ and $(w,v)$ are two non-loop edges whose half edges are paired at the vertex $v$, then we replace these two edges with a single edge $(u,w)$, which, in the case of a cellularly embedded graph, lies in an $\varepsilon$-neighbourhood of the original two edges.  In the case of a loop, we temporarily imagine an extra vertex of degree two on the loop, carry out the operation, and then remove the temporary vertex.  After this is done for all pairs in the partition at $v$, the   resulting configuration is called  a    {\em vertex state} at $v$.
If $G$ is an embedded graph and  $G_m$ is its canonically checkerboard coloured medial graph then we may distinguish among the vertex states at $v$ as in Figure~\ref{c1.vstate.f2}.  We call the three vertex states a {\em white smoothing}, a {\em black smoothing}, and a {\em crossing} as defined in the figure.   The white smoothing and black smoothing are known collectively as {\em smoothings}.

\begin{figure}[h]
\centering
\begin{tabular}{ccccccc}
\labellist \small\hair 2pt
\pinlabel {$v$}  at 36 22
\endlabellist
\includegraphics[scale=.45]{m1}
 & \quad \raisebox{5mm}{$\longrightarrow$} \quad  & \includegraphics[scale=.45]{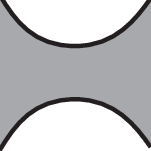} &\hspace{5mm}  & \includegraphics[scale=.45]{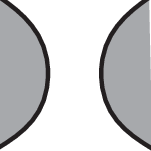} &\hspace{5mm}  &\includegraphics[scale=.45]{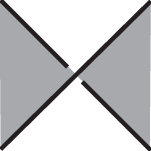} \\
in $G_m$ && white smoothing && black smoothing && crossing.
\\
 && (state weight $\alpha$) && (state weight $\beta$) && (state weight $\gamma$).
\end{tabular}
\caption{The three vertex states of a vertex $v$ of a canonically checkerboard coloured medial graph, and their state weights in the medial weight system $(\alpha, \beta, \gamma)$.}
\label{c1.vstate.f2}
\end{figure}

A {\em graph state} of  a 4-regular graph $F$ is a choice of vertex state at each of its vertices, which results in a set of curves.  The number of these curves is the number of components of the state.  If $F$ has no vertices, and so is either empty or a collection of free loops, then its unique graph state is just itself.
We let $\mathcal{S}(F)$ denote the set of graph states of $F$.
A graph state $s$ consists of a set of disjoint closed curves, and we refer to these as the {\em components} of the state, denoting the number of them by $c(s)$.

Let $G$ be an embedded graph, and $\alpha, \beta, \gamma\in \mathbb{C}$. (For simplicity, we will work with coefficients in $\mathbb{C}$, but our results may be adapted to other rings.)   The {\em medial weight system} for $G_m$ defined by  $(\alpha, \beta, \gamma)$ is the mapping $\omega:\mathcal{S}(G_m)  \rightarrow \mathbb{C}$ given by $\omega(s) := \prod_{v \in V(G_m)} {\omega(v,s)}$
where $(v,s)$ is the vertex state at the vertex $v$ in the state $s$;  and, with respect to the  canonical checkerboard colouring,  $\omega(v,s)=\alpha $  if $(v,s)$ is a white smoothing,     $\omega(v,s)=\beta $ if $(v,s)$ is a black smoothing, and $\omega(v,s)=\gamma $ if $(v,s)$ is a crossing state  (see Figure~\ref{c1.vstate.f2}).

 \begin{definition} \label{Qq}
   The \emph{topological transition polynomial} $Q(G, (\alpha, \beta, \gamma), t) \in \mathbb{C}[t] $ is defined by
\[
Q(G, (\alpha, \beta, \gamma), t) :=\sum_{s \in \mathcal{S}( G_m)} \omega( s )t^{c(s)},
\]
where $\omega$ is the medial weight system for $G_m$ defined by $(\alpha, \beta, \gamma)$.
\end{definition}

Since, as noted in Section \ref{Medial graphs}, the medial graph of an isolated vertex is a free loop, it follows that if $G$ has no edges, then $Q(G, (\alpha, \beta, \gamma), t) = t^{k(G)}$.

The topological transition polynomial was shown in \cite{EMMa} to admit a deletion-contraction-type reduction identity which is akin to that for the Tutte polynomial. Since it holds for vertex specific weights in \cite{EMMa}, it holds also for the version given in Definition \ref{Qq}, yielding the form below in Proposition~\ref{e.qdc}. This identity will be useful in subsequent inductive proofs.

\begin{proposition}\label{e.qdc}
Let $e$ be an edge of an embedded graph $G$.  Recalling that  edge contraction is defined as in Section~\ref{ss.contract}, we have
\begin{equation}
Q(G; (\alpha, \beta, \gamma), t)= \alpha Q(G/e; (\alpha, \beta, \gamma), t) + \beta Q(G-e; (\alpha, \beta, \gamma), t)+\gamma Q(G^{\tau(e)}/e; (\alpha, \beta, \gamma), t).
\end{equation}
 \end{proposition}

\subsection{$k$-valuations and the transition polynomial}

We now see that $Q(G)$ can be expressed as a sum over all $k$-valuations, and this, our first main theorem, will provide the unifying framework for similarities between $R(G)$ and $P(G)$.

\begin{theorem}\label{t.qk}
Let $G$ be an embedded graph, and  $k \in \mathbb{N}$. Then
\begin{equation}\label{e.qk}
  Q(G; ( \alpha, \beta , \gamma ), k) = \sum_{  \phi \in \mathcal{K}(G_m,k)  }   (\alpha+\beta+\gamma)^{\total(\phi)}  \alpha^{\white(\phi)}  \beta^{\black(\phi)}  \gamma^{\cross(\phi)}.
\end{equation}
\end{theorem}

\begin{proof}
We use induction on the number of edges of an embedded graph $G=(V,E)$. The claim  is easily verified  when $|E|=0$. Now assume that $|E|\neq 0$, and that the claim holds for all embedded graphs with fewer than $|E|$ edges.  For  $\phi \in \mathcal{K}(G_m,k)$ and $U\subseteq V(G_m)$, set
\[  \Theta(G,\phi, U) := \prod_{v_e\in U}   (\alpha+\beta+\gamma)^{\total(\phi,v_e)}  \alpha^{\white(\phi,v_e)}  \beta^{\black(\phi,v_e)}  \gamma^{\cross(\phi,v_e)}.  \]
Fixing some $e\in E$, we can write the right-hand side of Equation \eqref{e.qk} as
\begin{multline}\label{p.kval.e1}
\alpha \Big( \sum_{  \substack{\phi \in \mathcal{K}(G_m,k) \\ (\phi,v_e) \text{ tot. or wh.}   }} \Theta(G,\phi, V(G_m)\backslash \{e\}) \Big)
+\beta \Big( \sum_{  \substack{\phi \in \mathcal{K}(G_m,k) \\ (\phi,v_e) \text{ tot. or bl.}   }} \Theta(G,\phi, V(G_m)\backslash \{e\})\Big)
\\+\gamma \Big( \sum_{  \substack{\phi \in \mathcal{K}(G_m,k) \\ (\phi,v_e) \text{ tot. or cr.}   }} \Theta(G,\phi, V(G_m)\backslash \{e\})\Big).
\end{multline}
We show that the first term in \eqref{p.kval.e1} equals $\alpha Q(G/e; (\alpha,\beta,\gamma), k  )$.
Consider  $G$ locally at an edge $e$, and $G_m$ locally at the corresponding vertex $v_e$ as shown in Table~\ref{p.kval.f1}. Note that in the table we are not assuming that the vertices at the ends of $e$ are distinct or that they lie on the plane on which they are drawn, and so Table~\ref{p.kval.f1} and the following argument also includes the cases when $e$ is an orientable or non-orientable loop. 
Table~\ref{p.kval.f1} also shows $G/e$ and  $(G/e)_m$ at the corresponding locations. All the embedded graphs in the table are identical outside of the region shown. Let $\phi'$ be a $k$-valuation of $(G/e)_m$. Then both of the arcs of  $(G/e)_m$ shown in Table~\ref{p.kval.f1} are coloured with the same element $i$, or they are coloured by different elements $i$ and $j$.  The $k$-valuation $\phi'$ of $(G/e)_m$ naturally induces a $k$-valuation $\phi$ of $G_m$ in which $(\phi,v_e)$ is total or white. As this process is reversible, we have a bijection between
$\mathcal{K}((G/e)_m,k)$   and $ \{ \phi \in \mathcal{K}(G_m,k) \; |\; (\phi,v_e) \text{ total or white}\}$.
Under this bijection
\begin{equation}\label{p.kval.e2}
 \begin{aligned}
 \sum_{  \substack{\phi \in \mathcal{K}(G_m,k) \\ (\phi,v_e) \text{ tot. or wh.}   }} \Theta(G,\phi,  V(G_m)\backslash \{e\}) &=  \sum_{ \phi \in \mathcal{K}((G/e)_m,k) } \Theta(G/e,\phi, V((G/e)_m) \\& = Q(G/e; (\alpha,\beta,\gamma), k  ),
 \end{aligned}
  \end{equation}
where the second equality follows by the inductive hypothesis. Similar arguments give that the second and third sums in \eqref{p.kval.e1} equal $\beta Q(G-e)$ and $\gamma Q(G^{\tau(e)}/e)$ respectively. Thus Equation \eqref{p.kval.e1}  equals $\alpha Q(G/e) + \beta Q(G-e)+\gamma Q(G^{\tau(e)}/e)$, which
 by Proposition \eqref{e.qdc}, is   $Q(G)$.
\end{proof}

\begin{table}
\centering
\begin{tabular}{|c|c|c|c|c|}\hline
  & $G$ & $G/e$ & $G-e$   & $G^{\tau(e)}/e$\\ \hline
 \raisebox{5mm}{Graph} & \labellist \small\hair 2pt
\pinlabel {$e$}  at 64 39
\endlabellist\includegraphics[scale=.45]{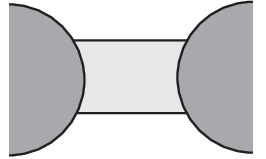} &\includegraphics[scale=.45]{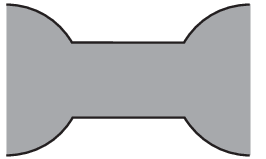}& \includegraphics[scale=.45]{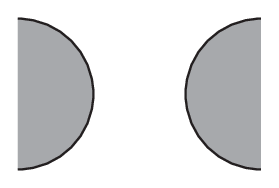}& \includegraphics[scale=.45]{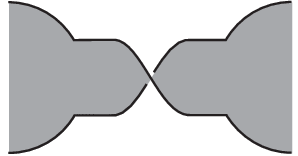} \\ \hline
  \raisebox{5mm}{Medial graph} &\labellist \small\hair 2pt
\pinlabel {$v_e$}  at 78 20
\endlabellist \includegraphics[scale=.4]{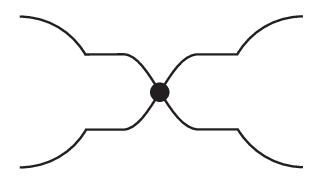} &\includegraphics[scale=.4]{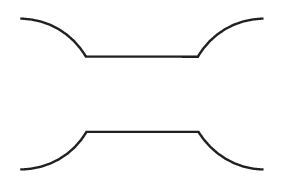}&\includegraphics[scale=.4]{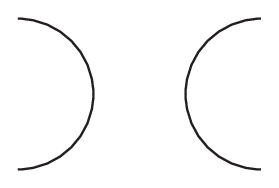}&  \includegraphics[scale=.4]{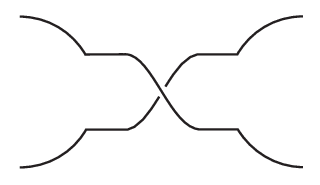} \\ \hline

\hline
\end{tabular}
\caption{Local differences in $G$, $G/e$, $G-e$, $G^{\tau(e)}/e$, and their medial graphs for a non-loop edge $e$}
\label{p.kval.f1}
\end{table}

\subsection{The ribbon graph and Penrose polynomials at positive integers}
In this section we show that Evaluations~\ref{t.KP}--\ref{sumcross}  arise as special cases of Theorem~\ref{t.qk}. For this we need the following result which shows that $R(G)$ and $P(G)$ are assimilated by the transition polynomial. This result provides the unifying framework for explaining the connections among the evaluations of $R$ and $P$.
\begin{proposition}[\cite{ES} and \cite{EMMb}]\label{p.qmbr}
If $G$ is an embedded graph. Then
\begin{align}
Q\big(G ;\big(\sqrt{y/x}, 1, 0\big), \sqrt{xy}\big) &= \label{e.qmbr0} x^{c(G)}  (\sqrt{y/x})^{v(G)}   R(G;  x+1, y, 1/\sqrt{xy})
 \\
\label{e.qmbr2}   Q(G ;(1, 0, -1), \lambda)    &= P(G; \lambda).
\end{align}
\end{proposition}

The specialisation  $R(G; x+1,y,1/\sqrt{xy})$ that appears in Equation~\eqref{e.qmbr0} is significant as it as appears in many results about the ribbon graph polynomial,  for example in duality relations  (see \cite{BR2,ES,Mo1}) and in connections to knot theory (see for example \cite{CP,Da,Mo1}). In addition, \mbox{$R(G; x+1,y,1/\sqrt{xy})$} has a deletion-contraction relation that reduces it to a linear combination of polynomials of edgeless ribbon graphs (see  \cite{Ch1}), which are easily computed. In contrast, the known deletion-contraction reduction for the three-variable polynomial $R(G;x,y,z)$  will only reduce it to a linear combination of polynomials of ribbon graphs in which every edge is a loop, and there is not currently an efficient method for computing these forms.

Proposition~\ref{p.qmbr} allows us to unify Evaluations~\ref{t.KP}--\ref{sumcross} and to explain their similarities  by showing that they are all special cases of Theorem~\ref{t.qk}.  In fact, this framework allows us to extend Korn and Pak's  Evaluation \ref{c.qk}  as seen  in \eqref{e.qk1} below.
\begin{evaluation}\label{c.qk}
Let $G$ be an embedded graph, and  $k \in \mathbb{N}$. Then
\begin{align}
\label{e.qk2} k^{c(G)}  R(G;k+1,k,1/k) &= \sum\limits_{  \phi \in \mathcal{R}(G_m,k)  }   2^{\total(\phi)}; \\
\label{e.qk3}  P(G;k) &= \sum\limits_{  \phi \in \mathcal{A}(G_m,k)  }   (-1)^{\cross(\phi)};\\
\label{e.qk1}     k^{c(G)}  b^{r(G)}   R\Big(G;  \frac{k+b}{b}, bk, \frac{1}{k}\Big)    &= \sum\limits_{  \phi \in \mathcal{R}(G_m,k)  }    (b+1)^{\total(\phi)}     b^{\white(\phi)}; \\
\label{e.qk4} T(G; k/b+1 ,kb+1) &=  \sum\limits_{  \phi \in \mathcal{R}(G_m,k)  }    (b+1)^{\total(\phi)}     b^{\white(\phi)}, \quad \text{when $G$ is plane}.
\end{align}
\end{evaluation}

\begin{proof} We prove Equation \eqref{e.qk1} first. For this, observe that by Equations~\eqref{e.qmbr0}, and~\eqref{e.qk},
\begin{align*}k^{c(G)}  b^{r(G)}   R(G;  (k+b)/b, bk, 1/k)&=Q(G; (b, 1, 0), k)
\\ & =\sum_{  \phi \in \mathcal{K}(G_m,k)  }   (b+1)^{\total(\phi)}  b^{\white(\phi)}  1^{\black(\phi)} 0^{\white(\phi)}.
\end{align*}
Note that $(b+1)^{\total(\phi)}  b^{\white(\phi)}  1^{\black(\phi)} 0^{\white(\phi)}=0$, whenever $\white(\phi)\neq 0$, and is $1$ otherwise. Therefore we can sum over  $ \mathcal{R}(G_m,k)$  and drop the $0^{\white(\phi)}$ term, giving the result.

Equation~\eqref{e.qk2}  follows from Equation  \eqref{e.qk1} by setting $b=1$.
 Equation~\eqref{e.qk3} is proved in a similar way to \eqref{e.qk1}, but using   \eqref{e.qmbr2} instead of \eqref{e.qmbr0}. We omit the details of this argument.
Equation~\eqref{e.qk4}  follows from Equation \eqref{e.qk1} as $T(G)$ and $R(G)$ agree on plane graphs as in Equation \eqref{TR1b}.  \end{proof}

\subsection{Duality and the Penrose polynomial at negative integers} \label{ss.dpn}
We have seen that Evaluations~\ref{t.KP}--\ref{sumcross} can be recovered from, and even extended by,  Theorem~\ref{t.qk}. We will now see that  Theorem~\ref{t.qk} can also be used to recover and extend Evaluation~\ref{t.total}.  To do this we exploit the fact that the transition polynomial is well behaved under both geometric and Petrie duality. This fact will allow us to obtain  Evaluation~\ref{t.total} from the special case of Theorem~\ref{t.qk} that appears  in Korn and Pak's Evaluation~\ref{t.KP}.
This is a very different approach from that used in the original proof of  Evaluation~\ref{t.total}  in \cite{EMMb} which relied upon a connection between $P(G)$ and  the circuit partition polynomial of \cite{EM98}.

It is well-known that the Tutte polynomial satisfies a geometric duality relation for plane graphs: $T(G;x,y)=T(G^*;y,x)$. An analogous geometric duality identity was determined for $Q(G)$ in \cite{ES}, and a much more general twisted duality identity for $Q(G)$ was given in \cite{EMMa}. In this paper we will use the following special cases of the twisted duality relation for $Q(G)$.
\begin{proposition}[\cite{EMMa,ES}]\label{t.Qdual}
Let $G$ be an embedded graph. Then
\begin{align}
\label{e.Qdual1} Q(G ; (\alpha , \beta, \gamma) , t )  &=  Q(G^* ; (  \beta, \alpha,  \gamma) , t ) ,\\
\label{e.Qdual2} Q(G ; (\alpha , \beta, \gamma) , t )  &=  Q(G^{\times} ; ( \gamma,  \beta, \alpha ) , t ) .
\end{align}
\end{proposition}

The following theorem is a central result in this paper. It is an easy consequence of Propositions~\ref{p.qmbr} and  \ref{t.Qdual}.  Its importance lies in that it provides a  useful connection between $R(G)$ and $P(G)$, and that it helps explain why they give such similar evaluations.  Furthermore, Theorem~\ref{PtoZ} will subsequently lead to additional identities and combinatorial interpretations of evaluations of both polynomials.

\begin{theorem}\label{PtoZ}
Let $G$ be an embedded graph. Then
\begin{equation}\label{eq.PtoZ}
P(G^*;\lambda) =  \lambda^{c(G^{\times})}  (-1)^{r(G^{\times})}   R(G^{\times};  1-\lambda, -\lambda, 1/\lambda).
\end{equation}
\end{theorem}

\begin{proof} We have
\begin{multline*}  P(G^*;\lambda)
= Q\left(G^* ;(1, 0, -1), \lambda \right)
= Q\left(G ;(0, 1, -1), \lambda \right)\\
= Q\left(G^{\times} ;(-1, 1, 0), \lambda \right)
=\lambda^{c(G^{\times})}  (-1)^{r(G^{\times})}   R(G^{\times};  1-\lambda, -\lambda, 1/\lambda),
\end{multline*}
where the first equality follows from  Equation~\eqref{e.qmbr2}, the second and third  from  Theorem~\ref{t.Qdual}, and the fourth  from Equation~\eqref{e.qmbr0}.
 \end{proof}

We now use Theorem \ref{PtoZ} to expand the class of graphs for which Evaluation \ref{t.total} applies, and to explain its connection to Theorem~\ref{t.qk}.
\begin{evaluation}\label{total2} Let $G$ be an embedded graph such that $G^{*\times}$ is orientable. Then
\begin{equation}\label{e.total2}
P(G;-k) = (-1)^{f(G)}  \sum_{  \phi \in \mathcal{P}(G_m,k)  }   2^{\total(\phi)}.
\end{equation}
\end{evaluation}

\begin{proof}
If $G$ is orientable, then the $z$'s in $R(G;x, y, z)$ always appear with even exponents. Thus, if $G$ is orientable,
\[
k^{c(G)}R(G;k+1, k, -1/k) =k^{c(G)}R(G;k+1, k, 1/k)
=  \sum\limits_{  \phi \in \mathcal{R}(G_m,k)  }   2^{\total(\phi)},
\]
where the last equality follows from Evaluation~\ref{t.KP}.
Using  Theorem~\ref{PtoZ}, when  $G^{*\times  }$ is orientable, we have
\[ P(G;-k)
=(-k)^{c(G^{*\times})}  (-1)^{r(G^{*\times})}   R(G^{*\times};  k+1, k, -1/k)
= (-1)^{f(G)} \sum\limits_{  \phi \in \mathcal{R}((G^{*\times})_m,k)  }   2^{\total(\phi)} ,\]
where the third equality follows since $v(G^{*\times})= v(G^{*}) = f(G)$.

To rewrite the $R$-permissible $k$-valuations of $(G^{*\times})_m,$ in terms of $P$-permissible $k$-valuations of $G_m$, consider Figure~\ref{c4.f.KP1}.
An edge $e$ in the arrow presentations of $G$ and $G^{*\times}$ is shown in Figures~\ref{c4.f.KP1a} and ~\ref{c4.f.KP1b}. The top row of the table in Figure~\ref{c4.f.KP1c} shows the vertex $v_e$ and boundary components of $G_m$  and $(G^{*\times})_m$, with the shading indicating the canonical checkerboard colouring at $v_e$. The first column shows, at $v_e$, the $P$-permissible $k$-valuations of $G_m$, and the second column shows the $R$-permissible $k$-valuations of $(G^{*\times})_m$. We see from the table that there is a bijection between the $P$-permissible $k$-valuations of $G_m$ and the $R$-permissible $k$-valuations of $(G^{*\times})_m$. Moreover, this bijection preserves the number of total vertices in the $k$-valuations. Thus,
\[ \sum\limits_{  \phi \in \mathcal{R}((G^{*\times})_m,k)  }   2^{\total(\phi)} = \sum_{  \phi \in \mathcal{P}(G_m,k)  }   2^{\total(\phi)}.\]
 The result then follows.
\end{proof}

\begin{figure}
\centering
\subfigure[An arrow presentation for an edge $e$ of $G$. ]{
\hspace{5mm} \includegraphics[height=15mm]{a1}\hspace{5mm}
\label{c4.f.KP1a}
}
\hspace{2cm}
\subfigure[The arrow presentation for the edge $e$ of $G^{*\times}$. ]{
\hspace{5mm} \includegraphics[height=15mm]{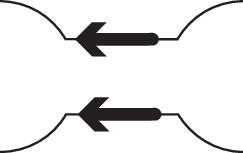}\hspace{5mm}
\label{c4.f.KP1b}
}
\subfigure[A table showing $G_m$ at $v_e$, and its $P$-permissible $k$-valuations in the first column; and $(G^{*\times})_m$ at $v_e$, and its $R$-permissible $k$-valuations in the second column.]{

\hspace{2cm}
\begin{tabular}{cccc}
& $G_m$&  \hspace{15mm} & $(G^{*\times})_m$ \\
&
\includegraphics[height=15mm]{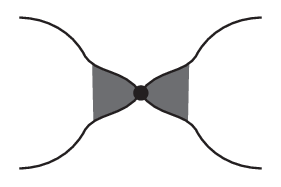}& & \includegraphics[height=15mm]{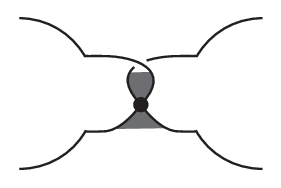} \\
&
\labellist \small\hair 2pt
\pinlabel {$i$}  at 3 9
\pinlabel {$i$}  at 3 81
\pinlabel {$i$}  at 131 81
\pinlabel {$i$}  at 131 9
\endlabellist
\includegraphics[height=15mm]{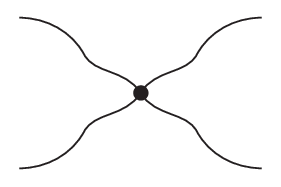}& &
\labellist \small\hair 2pt
\pinlabel {$i$}  at 3 9
\pinlabel {$i$}  at 3 81
\pinlabel {$i$}  at 131 81
\pinlabel {$i$}  at 131 9
\endlabellist\includegraphics[height=15mm]{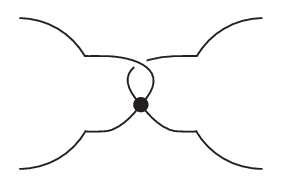} \\
&
\labellist \small\hair 2pt
\pinlabel {$j$}  at 3 9
\pinlabel {$i$}  at 3 81
\pinlabel {$i$}  at 131 81
\pinlabel {$j$}  at 131 9
\endlabellist
\includegraphics[height=15mm]{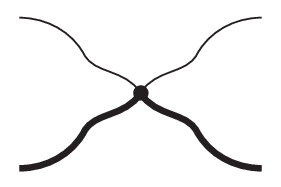}& &
\labellist \small\hair 2pt
\pinlabel {$j$}  at 3 9
\pinlabel {$i$}  at 3 81
\pinlabel {$i$}  at 131 81
\pinlabel {$j$}  at 131 9
\endlabellist
 \includegraphics[height=15mm]{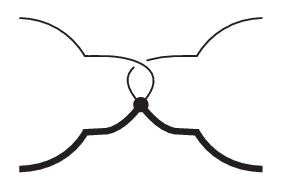} \\
&
\labellist \small\hair 2pt
\pinlabel {$j$}  at 3 9
\pinlabel {$i$}  at 3 81
\pinlabel {$j$}  at 131 81
\pinlabel {$i$}  at 131 9
\endlabellist
 \includegraphics[height=15mm]{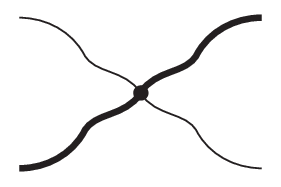}& &
 \labellist \small\hair 2pt
\pinlabel {$j$}  at 3 9
\pinlabel {$i$}  at 3 81
\pinlabel {$j$}  at 131 81
\pinlabel {$i$}  at 131 9
\endlabellist
  \includegraphics[height=15mm]{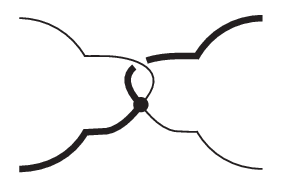} \\
  & $P$-permissible && $R$-permissible
\end{tabular}
\hspace{2cm}

\label{c4.f.KP1c}
}
\caption{Identifying the $P$-permissible $k$-valuations of $G_m$ and the $R$-permissible  $k$-valuations of $(G^{*\times})_m$.}
\label{c4.f.KP1}
\end{figure}

\begin{remark}
To see that Evaluation~\ref{total2} generalises  Evaluation~\ref{t.total}, observe that $G^{*\times}$ being orientable is a weaker condition than $G$ being orientable and checkerboard colourable.  For example, let $G$ be the
embedded graph consisting of one vertex and three non-orientable loops that meet the vertex in the cyclic order $(a\,b\,c\,a\,b\,c)$. Then $G^{*\times}$ is orientable (it is a plane 3-cycle), but $G$ is neither orientable nor checkerboard colourable.
On the other hand, if  $G$ is is orientable and checkerboard colourable, then  $G^*$ is is orientable and bipartite, so $G^{*\times}$ is orientable (every cycle of $G^*$ has even length, so receives an even number of half-twists).
\end{remark}

\begin{remark}
Although we have focused on $R(G)$ in this paper to facilitate connections with the current literature, Theorem \ref{PtoZ} and  several of the other results have much simpler formulations in terms of the topochromatic polynomial, which is just a reparameterisation of $R$:   $Z(G; x,y,z) = (xc/z)^{c(G)}y^{v(G)}R(G;(xz+y)/y, yz, 1/z)$ (see \cite{Mo1}).  For example, Equation~\eqref{eq.PtoZ} becomes simply  $P(G^*;\lambda) = Z(G^{\times};1, -1, \lambda)$.
\end{remark}

\section{Graph Colouring Connections}\label{s.furth}

\subsection{The chromatic polynomial and the Penrose polynomial}
In \cite{Ai97}, Aigner proved the following result which connects the Penrose polynomial of a plane graph $G$ and the chromatic polynomial $\chi(G^*, \lambda)$ of its geometric dual $G^*$.

\begin{theorem}[Aigner \cite{Ai97}]  \label{t.aigner}
Let $G$ be a plane graph, and $k\in \mathbb{N}$. Then  \begin{equation}\label{t.aigner.e1}\chi(G^*;k)\leq P(G;k).\end{equation}
\end{theorem}
In \cite{EMMb}, the authors extended Aigner's result by expressing, when  $G$ is plane, $P(G;\lambda)$ as a sum of chromatic polynomials in which $\chi(G^*;\lambda)$ appears as a single summand:
\begin{theorem}[\cite{EMMb}]\label{n3}
Let $G$ be  a plane graph.  Then
\begin{equation}\label{n3.e1}  P(G;\lambda) = \sum_{A\subseteq E(G)}  \chi ((   G^{\tau(A)}   )^*   ;\lambda)  .\end{equation}
\end{theorem}
We will now complete Theorems \ref{t.aigner} and \ref{n3} by showing that that Penrose polynomial of an arbitrary embedded graph can be written as sum of chromatic polynomials. This result is a consequence of the combinatorial interpretation of $P(G;k)$ in terms of $k$-valuations from Evaluation~\ref{sumcross}, and hence it follows from Theorem~\ref{t.qk}.
\begin{theorem}\label{c41.n5}
Let $G$ be an embedded graph. Then
\[  P(G;\lambda) = \sum_{A\subseteq E(G)}  (-1)^{ |A|}  \chi ((   G^{\tau(A)}   )^*   ;\lambda)  .\]
\end{theorem}
\begin{proof}
From Evaluation~\ref{sumcross}, for each $k \in \mathbb{N}$ we have
 \begin{equation*}\label{c41.n5.e1}
 P(G;k) = \sum\limits_{\phi\in \mathcal{A} (G_m,k)}
 (-1)^{\cross(\phi)}
 =\sum\limits_{X\subseteq V(G_m)}  (-1)^{|X|} \sum\limits_{
\substack{\phi\in \mathcal{A} (G_m,k) \\  (\phi,v)\text{ crossing} \iff v\in X}}  1.
\end{equation*}

 We now view $G$ as a ribbon graph  and construct a bijection from the set of admissible $k$-valuations of the medial graph $G_m$ to a certain set of colourings of the boundaries of the partial Petrials  of $G$.
 Let $A_{\phi}\subseteq E(G)$ be the set of edges corresponding to vertices of $G_m$ with crossings in the admissible $k$-valuation $\phi$.
  The cycles in $G_m$ (which are determined by the colours  in the $k$-valuation $\phi$) follow exactly the boundary components of the partial Petrial $G^{\tau(A_{\phi})}$. Moreover, the colours of the  cycles in the $k$-valuation induce a colouring of the boundary components of $G^{\tau(A_{\phi})}$.   We define a {\em proper boundary $k$-colouring } of a ribbon graph to be a map from its set of boundary components to the colours $ \{1,2,\ldots , k\}$ with the property that whenever  two boundary components share a common edge, they are assigned different colours. It is then clear that the map from $\phi$ to $G^{\tau(A_{\phi})}$ defines a bijection between the set of admissible $k$-valuations of $G_m$ and the set of proper boundary \mbox{$k$-colourings} of the partial Petrials of $G$. Moreover, this bijection identifies the vertices with crossing states with the edges in  $A_{\phi}$. Thus we have
  \begin{equation}\label{c41.n5.e2}
 \sum\limits_{X\subseteq V(G_m)}  (-1)^{|X|} \sum\limits_{
\substack{\phi\in \mathcal{A} (G_m,k) \\  (\phi,v)\text{ crossing} \iff v\in X}}  1
=
 \sum_{A\subseteq E(G)}(-1)^{|A|} \sum\limits_{
\substack{\text{proper boundary} \\ \text{ $k$-col. of  } G^{\tau(A)}}}  1.
\end{equation}
 A proper boundary $k$-colouring of $G^{\tau(A)}$  corresponds to a proper face $k$-colouring of $G^{\tau(A)}$,  and hence to a proper $k$-colouring of $(G^{\tau(A)})^*$.
 Since the chromatic polynomial counts proper $k$-colourings of a graph, we can then write \eqref{c41.n5.e2} as
\[
 \sum_{A\subseteq E(G)}(-1)^{|A|} \sum\limits_{
\substack{\text{proper face} \\ \text{ $k$-col. of } G^{\tau(A)}}}  1
=
 \sum_{A\subseteq E(G)}(-1)^{|A|} \sum\limits_{
\substack{\text{proper $k$-col.} \\ \text{of } (G^{\tau(A)})^*  }}  1
=
 \sum_{A\subseteq E(G)}  (-1)^{ |A|}  \chi ((   G^{\tau(A)}   )^*   ;k) .
\]
 The theorem then follows upon observing that $k$ is an arbitrary natural number.
\end{proof}

\begin{remark}
Theorem~\ref{n3} can be recovered from  Theorem~\ref{c41.n5} in the case of plane graphs, as,  if $|A|$ is odd, then $(G^{\tau(A)})^*  $ must contain a loop. To see why, draw $G$ on the plane, and use this to obtain a plane drawing of the  boundary cycles of $G^{\tau (A)}$ in the natural way. The Jordan Curve Theorem implies that two boundary cycles of $G^{\tau(A)}$  must cross each other an even number of times in this drawing. Thus, if $|A|$ is odd, some boundary cycle must cross itself, and this corresponds to a loop in $(G^{\tau(A)})^*  $.
\end{remark}

\subsection{Extending results to the ribbon graph polynomial}\label{ss.argp}

We have seen that the connection between $P(G^*)$ and    $R(G^{\times})$  of Theorem~\ref{PtoZ} can be used to reveal new properties of the Penrose polynomial.
We now use it to uncover properties of $R(G)$, beginning with a new formulation of the Four Colour Theorem.

\begin{theorem}
The following statements are equivalent:
\begin{enumerate}
\item  \label{fct1} the Four Colour Theorem is true;
\item\label{fct2}  for every connected, loopless plane graph $G$,  $(-1)^{v(G)}R(G^{\times} ; -2,-3,1/3)<0$;
\item\label{fct3} for every connected, loopless plane graph $G$,  $(-1)^{v(G)}R(G^{\times}; -3,-4,1/4)<0$.
\item \label{fct4}for every connected, loopless plane graph $G$,  $R(G^{\times} ; 3,2, -1/2)\neq 0$.
\end{enumerate}
\end{theorem}

\begin{proof}
Corollary~9 of \cite{Ai97} states  that the Four Colour Theorem is equivalent to $P(G;3)>0$ for all connected, bridgeless plane graphs $G$.  This is equivalent to $P(G^*;3)>0$ for all connected, loopless plane graphs $G$, which, by Theorem~\ref{PtoZ}, is equivalent to the inequality
$
(-1)^{r(G^{\times})}3^{c(G^{\times})} R(G^{\times} ; -2,-3,1/3)>0,
$
or, since $G$ is connected and has the same rank and number of vertices as $G^{\times}$,  this is equivalent to $(-1)^{v(G)}R(G^{\times} ; -2,-3,1/3, 1)<0$, which yields the equivalence between items \eqref{fct1} and \eqref{fct2} of the theorem.

The equivalence between items \eqref{fct1} and \eqref{fct3}  is shown in a similar way but using the fact from \cite{Ai97} that the Four Colour Theorem is equivalent to showing that $P(G;4)>0$ for all connected, bridgeless plane graphs $G$.

The equivalence between items \eqref{fct1} and \eqref{fct4} follows from \cite{Ja88}, Proposition~7, which says that if $G$ is plane and connected, then
$(-1)^{e(G)}(-1/2)^{v(G)-2} P(G;-2)$ is equal to the number of face 4-colourings of $G$.
\end{proof}

Another way to recover new results relies on the  partial duals of Chmutov \cite{Ch1} and  the more general twisted duals of  \cite{EMMa}. As we only use twisted duality explicitly in the proof of Corollary~\ref{t.interp2a}, to avoid a lengthy discussion we refer the reader to  \cite{EMMa} or  \cite{EMMb} for the relevant definitions and properties.
\begin{corollary}\label{t.interp2a}
Let $G$ be an embedded graph, and  $G^{\delta(A)}$ be its partial dual with respect to $A$.
\begin{enumerate}
\item \label{ttp3} If $G$ is an embedded graph,  then  \[(-1)^{r(G)}\lambda^{c(G)}  R(G;1-\lambda,- \lambda, 1/\lambda, 1)=    \sum_{A\subseteq E(G)}  (-1)^{|A|} \chi ((   G^{\delta(A)}   )^{\times} ;\lambda).\]

\item \label{ttp2} If $G$ is a plane graph, then
\[ (-1)^{r(G)}(1-x)^{c(G)} T(G;x,x)=     \sum_{A\subseteq E(G)}  (-1)^{|A|} \chi ((   G^{\delta(A)}   )^{\times} ;1-x).\]

\end{enumerate}
\end{corollary}
\begin{proof}
For Item \ref{ttp3}, observe that Theorems~\ref{PtoZ} and \ref{n3} give
\[ (-1)^{r(G)}\lambda^{c(G)}  R(G;1-\lambda,- \lambda, 1/\lambda)=  P((G^{\times} )^*; \la) = \sum_{A\subseteq E(G)}  \chi ((   (G^{\times} )^*)^{\tau (A)})^*  ;\lambda).\]
Rewriting the exponent, with $A^c:=E(G) \backslash A$, and using group relations for $\tau$ and $\delta$ from \cite{EMMb}, gives
\begin{equation*}
\delta\tau(E(G)) \tau(A)  \delta (E(G)) =
\delta \tau \delta \tau (A) \delta \delta \tau (A^c)=
\tau \delta(A) \tau(A^c)=
(G^{\delta(A)})^{\times},
  \end{equation*}
  and the result follows.

  Item \ref{ttp2} follows from Item \ref{ttp3} and Equation~\eqref{TR1b}.
  \end{proof}
The following  result is another  consequence of  Theorem~\ref{PtoZ}.

\begin{corollary}\label{t.interp2b} ~
Let $G$ be plane, connected, and cubic, and let $H= (G^*)^{\times}$.  Then
the number of edge $3$-colourings of $G$ is
\[ (-1)^{v(H)-1} 3 R(H;-2,-3,1/3) =  R(H;3,2,-1/2). \]
\end{corollary}
\begin{proof}
The result follows  easily from Theorem~\ref{PtoZ} and Equation \eqref{e.ptcol}.   For recovering the second expression observe that
$ v(G)/2 +v(H) = v(G)/2 +f(G) $.  As $G$ is cubic, $2e(G) = 3  v(G)$, so Euler's formula gives $f(G)  = 2 + v(G)/2$, and so  $v(G)/2 +f(G) = 2+v(G) \equiv 0 \mod 2 $.
  \end{proof}

\begin{remark}
Corollary~\ref{t.interp2a} and its proof relied upon twisted duality properties from \cite{EMMa}. Twisted duality, in fact, forms the foundation of most of the results in this paper. This is since Proposition~\ref{e.qdc}, and Equations \eqref{e.Qdual1} and \eqref{e.Qdual2}, which appear implicitly or explicitly in most of the proofs here, depend upon twisted duality.
\end{remark}

\section{Relating  graph polynomials through tensor products}\label{s.tp}

In   Sections  \ref{ss.dpn} and \ref{ss.argp} we obtained new combinatorial identities for $P(G)$ and $R(G)$. Since these results relied upon Theorem~\ref{PtoZ}, the evaluations for $R(G)$ were all restricted to the specialisation $R(G;  1-x, -x, 1/x)$.  In this section we outline a family of connections between the polynomials $P$ and $R$ that hold for other 1-variable specialisations of $R(G)$, and we describe how these provide new combinatorial interpretations of $R(G)$.   These connections between $P$ and $R$  arise  from a formula for $Q(G)$ for the tensor product of graphs.   There are various constructions that go by the name of tensor product in the literature.  Here we follow the terminology established by Brylawski in \cite{Bry82} in the context of identities for the classical Tutte polynomial, where the tensor product of two graphs essentially replaces each edge of one graph by the other graph.  This tensor product construction has been generalised to ribbon graphs in \cite{HM}. 

%We note here that there is more than one notion of a tensor product in the graph theory literature. The concept we use here is that which is closely related to parallel connections in matroids, as indicated in Figure~\ref{f.tens1}.

We start by illustrating the general principle behind our new identities with a special case. Let $G$ be a cellularly  embedded graph, and let  $D(G)$  be the cellularly embedded graph obtained from $G$ by doubling all the edges, that is, for each edge $e$ of $G$, by embedding a new edge parallel to $e$ so that the two edges bound a face of degree $2$. By comparing the calculations of the transition polynomial at an edge of $G$, and the corresponding pair of edges in $D(G)$, we see that
\begin{equation}\label{e.dbl1}
Q(G;  (\alpha^2t+2\alpha(\beta +\gamma) , \beta^2+\gamma^2, 2\beta \gamma  ) ,t)  = Q(D(G);  (\alpha
 , \beta, \gamma  ) ,t).
 \end{equation}
Setting $(\alpha
 , \beta, \gamma  )=(1,0,-1) $ and using  Proposition~\ref{p.qmbr} then yields that
\begin{equation}\label{e.dbl2} t^{c(G)}  (t-2)^{r(G)}   R(G;  (2(t-1)/(t-2), t(t-2), 1/t)  =
P(D(G), t).
 \end{equation}
Thus, we have a new connection between $R$ and $P$.  Combining it with previous evaluations of the Penrose polynomial gives new interpretations of $R(G)$. For example,
\[ k^{c(G)}  (k-2)^{r(G)}  R(G; 2(k-1)/(k-2),  k(k-2) , 1/k) =    \sum\limits_{  \phi \in \mathcal{A}(D(G)_m,k)  }   (-1)^{\cross(\phi)}  ;\]
\[\lambda^{c(G)}  (\lambda-2)^{r(G)}  R(G; 2(\lambda-1)/(\lambda-2),  \lambda(\lambda-2) , 1/\lambda) =   \sum\limits_{A\subseteq E(D(G))}  (-1)^{ |A|}  \chi (   (D(G)^{\tau(A)}   )^*   ;\lambda) ;\]
and  if $G^{*\times}$ is orientable
\[k^{c(G)}  (k-2)^{r(G)}  R(G; 2(k-1)/(k-2),  k(k-2) , 1/k )=  (-1)^{f(G)}  \sum\limits_{  \phi \in \mathcal{P}(D(G)_m,k)  }   2^{\total(\phi)}.    \]
Note that by Equation~\eqref{TR1b}, these evaluations  give identities for the Tutte polynomial of a plane graph.

We will now fully generalise this construction and process.  Instead of just doubling edges, we will be able to replace the edges of $G$ by any graph, using the tensor operation for embedded graphs from \cite{HM}.  We will then use the same process of examining terms and appropriately adjusting the weights to find a general formula for the transition polynomial of a tensor product.  Finally, this formula, combined with the previous relations among the various polynomials, will give the desired family of new connections between $P$ and $R$ and hence new evaluations.

Just as with abstract graphs, the intuitive idea of a tensor product $G\otimes_{\varphi} H$ of two embedded graphs $G$ and $H$  is to replace each edge of $G$ with a copy of $H-e$.  We will define tensor products in terms of ribbon graphs. It is an easy exercise to reformulate this definition in terms of cellularly embedded graphs or arrow presentations.
\begin{definition}\label{d.tp1}
Let  $G$ and $H$  be ribbon graphs and $e$ be a distinguished non-loop  edge of $H$.
Let $\{\beta^1, \beta^2\} = \{ e\} \cap V(H)$ (so each $\beta^i$ is an arc where the edge $e$ is attached to its incident vertices). Similarly, for each $f\in E(G)$, let
 $\{\alpha_f^1, \alpha_f^2\} = \{ f\} \cap V(G)$.
Furthermore, let  $\varphi:=\{ \varphi_f \}_{f\in E(G)}$
 be a family of homeomorphisms such that for each $i$ and $f$ we have $\varphi_f: \{f\}\rightarrow \{e\}$ and $\varphi_f(\alpha_f^i) = \beta^j$ where $j\in \{1,2\}$ (since $\varphi$ is a homeomorphism, this means that $\varphi_f(\alpha_f^1)\neq \varphi_f(\alpha_f^2)$).
 We call $\varphi$ the family of {\em identifying maps}.
Then the {\em tensor product},  $G\otimes_{\varphi} H$, is the ribbon graph formed as follows:
\begin{enumerate}
\item  For each $f\in E(G)$, take a copy of $H$ and identify  $f$ and $e$  using $\varphi_f$.
\item In the resulting complex, for each $f$, delete the image $\varphi(f)=e$ of the identified edges except for the arcs $\varphi_f(\alpha_f^i) = \beta^i$.   Note that, for $i=1,2$   this merges two vertex discs, one from $G$ and one from $H$, that intersect in the  arc $\varphi_f(\alpha_f^i ) = \beta^j$ to form a single vertex.

    %Then, for $i=1,2$  merge the two vertex discs, one from $G$ and one from $H$, that intersect in the  arc $\varphi_f(\alpha_f^i ) = \beta^j$ to form a single vertex.
\end{enumerate}
\end{definition}
 An example of a tensor product of plane graphs is given in Figure~\ref{f.tens1}. As another example, $G\otimes C_3$ is the graph obtained by subdividing each edge of $G$; and $D(G)$ is $G\otimes \theta$, where $\theta$ is the plane graph consisting of two vertices joined by three parallel edges.  Figure~\ref{c1.loops} shows $G$ locally at $f$ and the corresponding location in $G\otimes_{\varphi} H$, in which $f$ is ``replaced'' by a copy of $H-e$.  Observe that up to equivalence of ribbon graphs, for each $f\in E(G)$, there are four possible ways to  replace $f$ with a  copy of $H-e$.

\begin{figure}
\centering
\subfigure[A plane graph $G$. ]{
\quad\quad \includegraphics[scale=0.6]{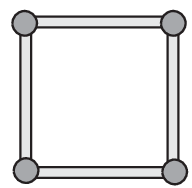}\quad\quad
\label{f.tens1a}
}
\hspace{1cm}
\subfigure[A plane graph $H$.]{
\labellist \small\hair 2pt
\pinlabel {$e$}  at 1 56
\endlabellist
\quad\quad\includegraphics[scale=0.6]{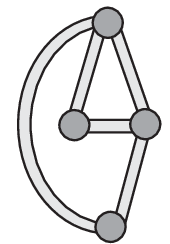}\quad\quad
\label{f.tens1b}}
\hspace{1cm}
\subfigure[A tensor product \mbox{$G\otimes_{\varphi} H$}.]{
\quad\quad\includegraphics[scale=0.6]{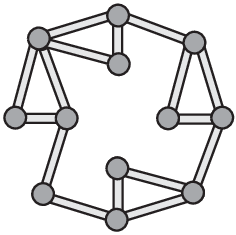}\quad\quad
\label{f.tens1c}
}
\caption{An example of a tensor product $G\otimes_{\varphi} H$. The map $\varphi$ determines how the edge $e$ of $H$ is identified with each edges of $G$. }
\label{f.tens1}
\end{figure}

Identities for tensor products of graphs, also called parallel connections,  have been given for various forms of the Tutte polynomial, for example in  \cite{Bry82,HM,OW92,Tra05,Woo02}.  Our central result  gives the corresponding formula for the transition polynomial in Theorem \ref{t.tensor}.

\begin{theorem}\label{t.tensor}
Let $G$ be an embedded graph, $H$ be an embedded  graph with a distinguished non-loop  edge $e$, and $G\otimes_{\varphi} H$ be their tensor product determined by the identifying maps $\varphi$. Then, for $t \in \mathbb{C} \backslash \{0, 1, -2\} $,
\[
Q(G\otimes_{\varphi} H; (\alpha, \beta, \gamma), t  )   =  Q(G; (\kappa, \lambda, \mu), t  ),
\]
where $\kappa$, $\lambda$, and $\mu$ are the unique solutions to
\begin{align}
\label{t.tensor.e1} t \kappa+t^2 \lambda+t \mu&= Q( H-e;   (\alpha, \beta, \gamma), t  ) \\
\label{t.tensor.e2} t^2 \kappa+t \lambda+t \mu&= Q( H/e;   (\alpha, \beta, \gamma), t  ) \\
\label{t.tensor.e3} t \kappa +t \lambda+t^2 \mu&= Q( H^{\tau(e)}/e;   (\alpha, \beta, \gamma), t  ).
\end{align}
\end{theorem}
\begin{proof}
We work in the language of ribbon graphs. We draw medial graphs by placing vertices in the centres of the edges of the original graph, and then drawing the edges of the medial graphs so that they pass through the endpoints of intersections between edge discs and vertex discs of the original graph, namely, the points $\{a_f,b_f,c_f,d_f\}$ shown in Figure \ref {c1.loops.a}. Likewise, we suppose that the cycles in the states of a medial graph also pass through these points. We choose an orientation for each edge $f$ of $G$, and arrange the labels $\{a_f,b_f,c_f,d_f\}$ so that the labeled points are met in the order $a_f,b_f,c_f,d_f$ when following the orientation of the edge, and furthermore, the points are labeled so that $a_f$ and $b_f$ are the end points of  one interval of intersection between $f$ and a vertex, while $c_f$ and $d_f$ are the end points of the other (see  Figure~\ref{c1.loops.a}).  This labelling of points in $G$ then induces a labelling of points on the boundaries of  $u_f$ and $v_f$  in  $G\otimes_{\varphi} H $ as indicated in  Figure~\ref{c1.loops.b}; and also a labelling of  points  on the vertices of  copies of $H$, $H-e$, $H/e$, and $H^{\tau(e)}/e$.

The proof is structured as follows.  We first show that the graphs states of $(G\otimes_{\varphi} H)_m$ may be viewed as graph states of $G_m$ with a number of additional cycles arising from the copies of $(H-e)_m$. The former correspond to cycles of $(G\otimes_{\varphi} H)_m$ that pass through the points   in $\{a_f,b_f,c_f,d_f \,|\, f\in E(G)\} $.  To enumerate the latter, at each edge $f$ of $G$ we consider the cycles of $(H-e)_m$ that do not pass through $\{a_f,b_f,c_f,d_f\}$.  These can be expressed in terms of $Q( H-e;   (\alpha, \beta, \gamma), t  )$, $Q( H/e;   (\alpha, \beta, \gamma), t  )$, and $Q( H^{\tau(e)}/e;   (\alpha, \beta, \gamma), t  )$, by solving for $\kappa, \lambda, \mu$ in the equations above.  This gives the appropriate weights to substitute into the computation of $Q(G)$, from which the result will follow.

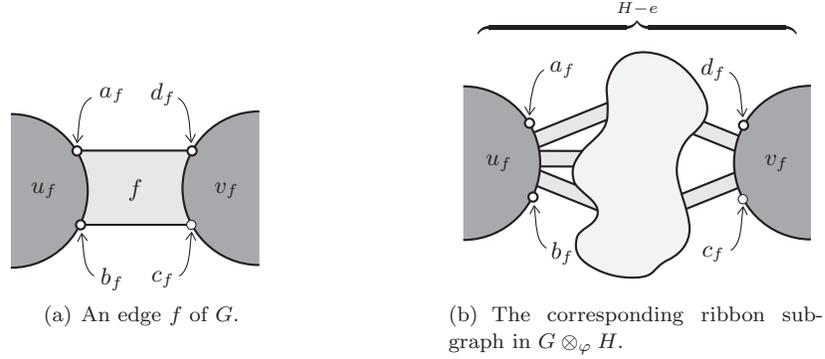
\begin{figure}
\centering
\subfigure[An edge $f$ of $G$.]{
\labellist \small\hair 2pt
\pinlabel {$f$}  at   63 47
\pinlabel {$u_f$} at  20 47
\pinlabel {$v_f$}  at   106 47
\pinlabel {$a_f$}  at    52 92
\pinlabel {$b_f$}  at    52 5
\pinlabel {$c_f$}  at    76 5
\pinlabel {$d_f$}  at  76 92
\endlabellist
\includegraphics[scale=.80]{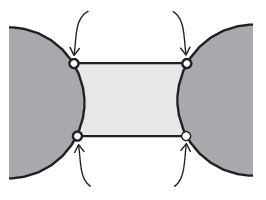}
\label{c1.loops.a}
}
\hspace{20mm}
\subfigure[The corresponding ribbon subgraph in $G\otimes_{\varphi} H$.]{
\labellist \small\hair 2pt
\pinlabel {$u_f$} at  20 60
\pinlabel {$v_f$}  at   151 60
\pinlabel {$a_f$}  at    50 104
\pinlabel {$b_f$}  at    50 17
\pinlabel {$c_f$}  at    121 17
\pinlabel {$d_f$}  at  121 104
\pinlabel {$\overbrace{\quad\quad\quad\quad\quad\quad\quad\quad\quad\quad\quad\quad\quad}^{H-e}$}  at  85 128
\endlabellist
\includegraphics[scale=.80]{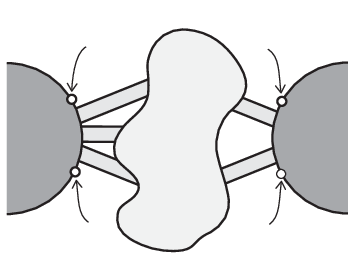}
\label{c1.loops.b}
}
\caption{A labelling used in the proof of Theorem~\ref{t.tensor}.}
\label{c1.loops}
\end{figure}

We begin by partitioning the graph states of $(G\otimes_{\varphi} H)_m$ according to how they pass through the set of points $\{a_f,b_f,c_f,d_f \,|\, f\in E(G)\} $. A cycle in a graph state of $(G\otimes_{\varphi} H)_m$ that passes through $a_f$ for some $f$ must enter the copy of $(H-e)_m$ and then exit it by passing through one of $b_f$, $c_f$ or  $d_f$, thus pairing $a_f$ with one of the other points.  The remaining two points are necessarily similarly paired in the graph state.  Thus, in any given state of $(G\otimes_{\varphi} H)_m$, for each edge $f$, one of the three situations occurs: $a_f$ is paired with $b_f$, and $c_f$ with  $d_f$; or $a_f$ is paired with   $c_f$, and $b_f$ with $d_f$; or  $a_f$ is paired with  $d_f$, and $b_f$ with  $c_f$.  We consider two graph states of $(G\otimes_{\varphi} H)_m$ to be in the same set of the partition if they have the same pairings for each $f$.   Let $\mathcal{X}$ denote this partition. Then
\begin{equation}\label{ep1}
Q(G\otimes_{\varphi} H; (\alpha, \beta, \gamma), t  )   =       \sum_{s\in \mathcal{S}( (G\otimes_{\varphi} H)_m  )} \omega(s)   t^{c(s)} =   \sum_{X\in \mathcal{X}}  \sum_{s\in X} \omega(s)   t^{c(s)} .
\end{equation}

Because each set $X$ is determined by pairings in $\{a_f,b_f,c_f,d_f \,|\, f\in E(G)\}$, and likewise each graph state of $G_m$ is determined by these same pairings, each set $X$ corresponds to a unique graph state of $G_m$.  The set of cycles in any graph state $s\in X$ can be partitioned into the set of cycles that pass through points in $\{a_f,b_f,c_f,d_f \,|\, f\in E(G)\} $, and the set cycles that do not.  The number of cycles in $s\in X$ that pass through some point in $\{a_f,b_f,c_f,d_f \,|\, f\in E(G)\}$ is exactly the number of cycles in this graph state of $G_m$ corresponding to $X$.  The number of cycles  in each $s\in X$ which do not pass through any of these points can be obtained by, for each $f\in E(G)$, considering the cycles in all the graph states of $(H-e)_m$ in which  $a_f,b_f,c_f,d_f$ are paired in the same way as in $s$, then counting all cycles except of those that pass though any of the points $a_f,b_f,c_f,d_f$ on $(H-e)_m$.

By these observations we have
\begin{equation}\label{ep2}
 \sum_{X\in \mathcal{X}}  \sum_{s\in X} \omega(s)   t^{c(s)}  =    \sum_{s\in \mathcal{S}( G)  }   t^{c(s)}   \prod_{f\in E(G)}    \sum_{\sigma \in \mathcal{ Y}_f(s)}  \omega(\sigma) t^{b(\sigma)} ,
\end{equation}
where $\mathcal{ Y}_f(s)$ is the set of graph states of $(H-e)_m$ in which  $a_f,b_f,c_f,d_f$ are paired in the same way as in $s$, and $b(\sigma)$ is the number of cycles in $\sigma$ that do not contain a point $a_f$, $b_f$, $c_f$, or $d_f$.

We now turn our attention to determining   $\sum_{\sigma \in \mathcal{ Y}_f(s)}  \omega(\sigma) t^{b(\sigma)}$. For this we split up the terms of the transition polynomial $Q (H-e)$ according to how the cycles pass through the points $a_f,b_f,c_f,d_f$. Let
\begin{itemize}
\item $Q_{=} (H-e)  $  denote the terms of $Q (H-e ;(\alpha, \beta, \gamma), t  )  $ which arise from the states of $(H-e)_m$ in which there is a cycle containing the points $a_f, b_f, c_f, d_f$ in that order;

\item $Q_{\shortparallel} (H-e)  $  denote the terms of $Q (H-e ;(\alpha, \beta, \gamma), t  )  $ which arise from the states of $(H-e)_m$ in which  there is a cycle containing the points $a_f, b_f$, and a different cycle containing the points  $c_f, d_f$;

\item $Q_{\times} (H-e)  $  denote the terms of $Q (H-e ;(\alpha, \beta, \gamma), t  )  $ which arise from the states of $(H-e)_m$ in which  there is a cycle containing the points $a_f, b_f, d_f, c_f$ in that order.
\end{itemize}
The partitioning of terms used in the definition of  $Q_{\times} (H-e) $ is illustrated in Figures~\ref{f.qdefs.a} and~\ref{f.qdefs.d}.
\begin{figure}
\centering
\subfigure[A copy of $H-e$.]{
\labellist \small\hair 2pt
\pinlabel {$a_f$}  at   63 104
\pinlabel {$b_f$}  at    62 18
\pinlabel {$c_f$}  at    134 18
\pinlabel {$d_f$}  at  134 104
\endlabellist
\includegraphics[scale=.80]{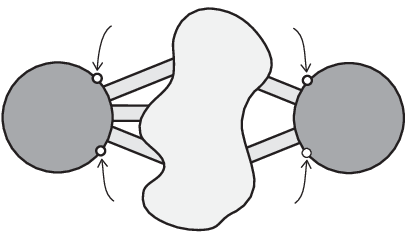}
\label{f.qdefs.a}
}
\hspace{10mm}
\subfigure[A graph state of $(H-e)_m$  that contributes to  $Q_{\times} (H-e) $ drawn on $H-e$.]{
\labellist \small\hair 2pt
\pinlabel {$a_f$}  at   63 104
\pinlabel {$b_f$}  at    62 18
\pinlabel {$c_f$}  at    134 18
\pinlabel {$d_f$}  at  134 104
\endlabellist
\includegraphics[scale=.80]{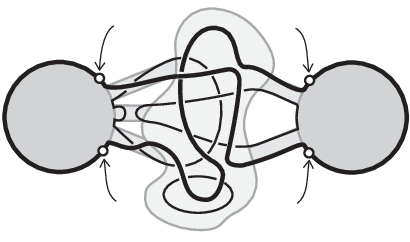}
\label{f.qdefs.d}
}

\vspace{5mm}

\subfigure[A graph state of $(G\otimes_{\varphi} H)_m$ drawn on $G\otimes_{\varphi} H$ locally at a copy of $H-e$.]{
\labellist \small\hair 2pt
\pinlabel {$u_f$} at  25 60
\pinlabel {$v_f$}  at   147 60
\pinlabel {$a_f$}  at    50 104
\pinlabel {$b_f$}  at    50 17
\pinlabel {$c_f$}  at    121 17
\pinlabel {$d_f$}  at  121 104
\pinlabel {$\overbrace{\quad\quad\quad\quad\quad\quad\quad\quad\quad\quad\quad}^{H-e}$}  at  85 128
\endlabellist
\includegraphics[scale=.80]{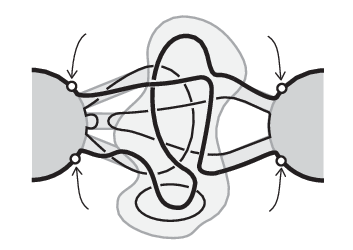}
\label{f.qdefs.e}
}
\hspace{10mm}
\subfigure[A graph state of $G_m$ drawn on $G$ locally at an edge $f$.]{
\labellist \small\hair 2pt
\pinlabel {$u_f$} at  25 60
\pinlabel {$v_f$}  at   147 60
\pinlabel {$a_f$}  at    50 104
\pinlabel {$b_f$}  at    50 17
\pinlabel {$c_f$}  at    121 17
\pinlabel {$d_f$}  at  121 104
\endlabellist
\includegraphics[scale=.80]{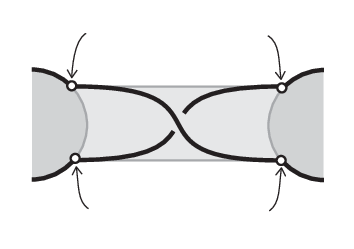}
\label{f.qdefs.f}
}
\caption{Illustrating the definition of $Q_{\times} (H-e) $, and the connection between states of $(H-e)_m$, $(G\otimes_{\varphi} H)_m$, and  $G_m$.}
\label{f.qdefs}
\end{figure}
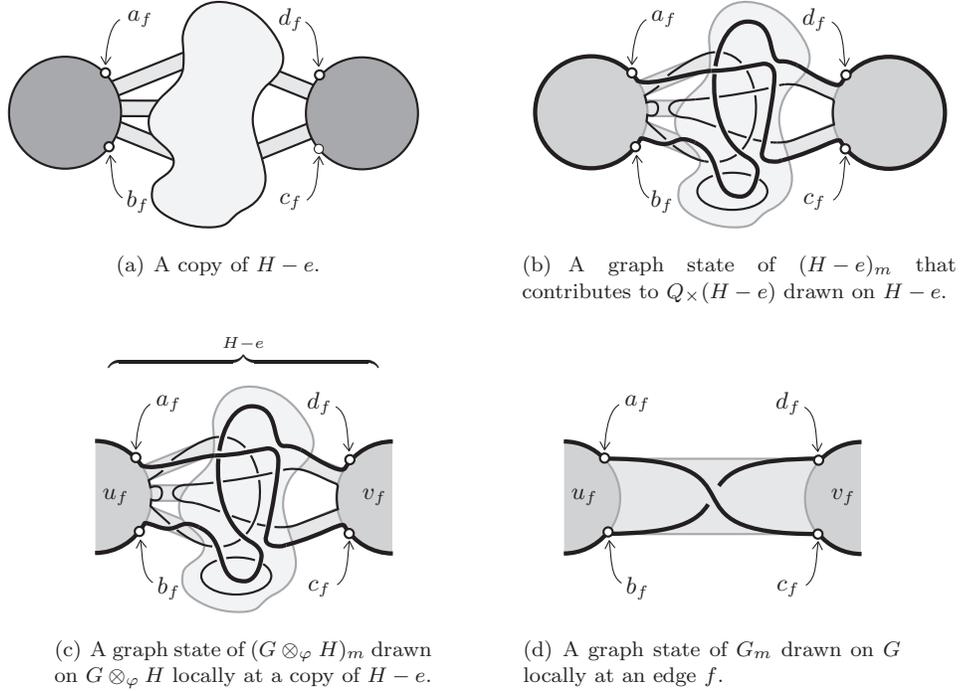
We have
\begin{enumerate}
\item \label{i1} $  \sum_{\sigma \in \mathcal{ Y}_f(s)}  \omega(\sigma) t^{b(\sigma)}=   t^{-1}Q_{=} (H-e)$,  when $s$  pairs $a_f$  with $d_f$, and $b_f$ with $c_f$ (i.e., $s$ has a white smoothing at $f$ );
\item \label{i2} $  \sum_{\sigma \in \mathcal{ Y}_f(s)}  \omega(\sigma) t^{b(\sigma)}=   t^{-2}Q_{\shortparallel} (H-e)  $
    when $s$  pairs $a_f$  with $b_f$, and $c_f$ with $d_f$ (i.e., $s$ has a black smoothing at $f$ ); and
\item \label{i3} $  \sum_{\sigma \in \mathcal{ Y}_f(s)}  \omega(\sigma) t^{b(\sigma)}=  t^{-1}Q_{\times} (H-e)  $ when $s$,  pairs $a_f$  with $c_f$ and $b_f$ with $d_f$ (i.e., $s$ has a crossing at $f$ ).
\end{enumerate}
 Figure~\ref{f.qdefs} provides an illustration of Item \eqref{i3} above: consider the graph state of $(G\otimes_{\varphi} H)_m$ shown locally at a copy of $H-e$ in Figure~\ref{f.qdefs.e}.  Figure~\ref{f.qdefs.d} shows the corresponding graph state in $(H-e)_m$. Then the contribution  $\omega(\sigma) t^{b(\sigma)}$ made by the graph state of Figure~\ref{f.qdefs.e} is equal to  that made by the graph state in Figure~\ref{f.qdefs.d} except we have over counted the cycles (specifically the one containing the points $a_f, b_f, d_f, c_f$).   Item \eqref{i3} is readily seen to follow. The cycles of $(G\otimes_{\varphi} H)_m$ containing the labels are counted by the cycles in the vertex states of $G_m$ as in Figure~\ref{f.qdefs.f} which shows the vertex state of $G_m$ corresponding to Figure~\ref{f.qdefs.e}.

Returning to the proof, combining items \eqref{i1}--\eqref{i3} with Equations~\eqref{ep1} and~\eqref{ep2} gives
\begin{multline*}\label{splits}
 Q(G\otimes_{\varphi} H; (\alpha, \beta, \gamma), t  )  \\=
 \sum_{s\in \mathcal{S}( G)  }  { t^{c(s)}  (t^{-1}Q_{=} (H-e))^{\ws(s)}  (t^{-2}Q_{\shortparallel} (H-e))^{\bs(s)}  (t^{-1}Q_{\times} (H-e))  )^{\cs(s)}}   \\=   Q(G; (  t^{-1}Q_{=} (H-e),  t^{-2}Q_{\shortparallel} (H-e) , t^{-1}Q_{\times} (H-e)), t),
   \end{multline*}
where, in the middle equation, $\ws(s)$, $\bs(s)$, and $\cs(s)$ are respectively the numbers of white, black, and crossing vertex states in the graph state $s$.

It remains to determine $Q_{=} (H-e)$,  $Q_{\shortparallel} (H-e)$, and  $Q_{\times} (H-e)$.
To do this we define $Q_{=} (H/e)$,  $Q_{\shortparallel} (H/e)$, $Q_{\times} (H/e)$, $Q_{=} (H^{\tau(e)}/e)$, $Q_{\shortparallel} (H^{\tau(e)}/e)$, and $Q_{\times} (H^{\tau(e)}/e)$ in a similar way to the $H-e$ case above.
$Q_{=} (H/e)$,  $Q_{\shortparallel} (H/e)$, $Q_{\times} (H/e)$ are the terms of $Q(H/e)$ in which, respectively, $ad$ appear on a cycle, and $bc$ appear on a different cycle; $abcd$ appear on a cycle in that order; and $adbc$ appear on a cycle in that order.
Also, $Q_{=} (H^{\tau(e)}/e)$, $Q_{\shortparallel} (H^{\tau(e)}/e)$, and $Q_{\times} (H^{\tau(e)}/e)$ are the terms of $Q(H^{\tau(e)}/e)$ in which, respectively,  the points $adbc$ appear on a cycle in that order; $abdc$ appear on a cycle in that order; and  $ac$ appear on a cycle, and $bd$ appear on a different cycle.

We have
\begin{equation}\label{tpproof1}
\begin{split}
 Q(H-e; (\alpha, \beta, \gamma), t  ) &= Q_{=} (H-e)+Q_{\shortparallel} (H-e)+Q_{\times} (H-e), \\
  Q(H/e; (\alpha, \beta, \gamma), t  ) &= Q_{=} (H/e)+Q_{\shortparallel} (H/e)+Q_{\times} (H/e), \\
  Q(H^{\tau(e)}/e; (\alpha, \beta, \gamma), t  ) &= Q_{=} (H^{\tau(e)}/e)+Q_{\shortparallel} (H^{\tau(e)}/e)+Q_{\times} (H^{\tau(e)}/e).
\end{split}
\end{equation}
It is easily checked that
\begin{align*}
Q_{=} (H/e) =& tQ_{=} (H-e) , &
 Q_{\shortparallel} (H/e) =& t^{-1}Q_{\shortparallel} (H-e),&
  Q_{\times} (H/e) =& Q_{\times} (H-e), \\
  Q_{=} (H^{\tau(e)}/e) =& Q_{=} (H-e) , &
 Q_{\shortparallel} (H^{\tau(e)}/e) =& t^{-1}Q_{\shortparallel} (H-e),&
  Q_{\times} (H^{\tau(e)}/e) =& tQ_{\times} (H-e).
\end{align*}
Writing $t^{-1}Q_{=} (H-e)=\kappa$, $t^{-2}Q_{\shortparallel} (H-e)=\lambda$, and $t^{-1}Q_{\times} (H-e)=\mu$, then  substituting these into (\ref{tpproof1}) gives the system of equations in (\ref{t.tensor.e1}),  (\ref{t.tensor.e2}), (\ref{t.tensor.e3}) to solve for  $Q_{=} (H-e)$,  $Q_{\shortparallel} (H-e)$, and  $Q_{\times} (H-e)$. The theorem then follows since the determinant of the resulting system of equations has roots 0, 1 and -2.
\end{proof}

\begin{remark} Theorem \ref{t.tensor} gives unique solutions for $t \in \mathbb{C} \backslash \{0, 1, -2\} $, but what happens when $t\in \{0, 1, -2\} $?  If $t=0$, then $Q(G\otimes_{\varphi} H)=Q(H)=0$.  However, a factor of $t$ may be cancelled in all of equations (\ref{t.tensor.e1}),  (\ref{t.tensor.e2}), (\ref{t.tensor.e3}) since the right hand sides have no constant terms, and then again a unique solution results.
If $t=1$, then by Theorem \ref{t.qk}, it follows that $Q(G,(\alpha, \beta, \gamma), 1  )= c^{e(G)}$ for some constant $c$, and thus     $Q(G\otimes_{\varphi} H; (\alpha, \beta, \gamma), t  )$ is just $c^{e(G) (e(H)-1))}$.  The case when $t=-2$ is more interesting. For example, if $E(G)\neq \emptyset$ then it is easy to show (by induction on the number of edges) that $Q(G;(\alpha, \alpha,\alpha) ,-2)=0$, for all values of $\alpha$. Thus if $H$ has at least two edges, all possible values of $\kappa$, $\lambda$, and $\mu$ give a solution to the system of equations in Theorem~\ref{t.tensor}. It is then easy to find examples for which  $Q(G\otimes_{\varphi} H;(\alpha, \alpha,\alpha) ,-2)=0$ but   $Q(G; (\kappa, \lambda, \mu), -2  ) \neq 0$. Thus Theorem \ref{t.tensor} does not hold when $t=-2$. Finding an analogue of  Theorem \ref{t.tensor} for when $t=-2$ is an interesting open problem.  
\end{remark}

\begin{remark}
Although, as mentioned following Definition \ref{Qq}, we did not use the version of $Q$ from \cite{EMMa} with locally specified vertex weights here, Theorem \ref{t.tensor} readily adapts to this more general setting.  In fact, using the same techniques as in the proof of Theorem \ref{t.tensor}, it is possible to determine the necessary weights at each vertex of $G_m$ to accommodate the insertion of a different graph at each edge of $G$ (or even at just a single edge) instead of inserting the same graph $H$ at all of them.
\end{remark}

At each edge $f$ of $G$, up to equivalence there are four possible ways to identify an edge $f$ of $G$ and the distinguished edge $e$ of $H$.  These choices are determined by the identifying maps $\varphi$. This means  there are up to $4^{e(G)}$  different tensor products of $G$ and $H$. A consequence of Theorem~\ref{t.tensor} is that the transition polynomial, and hence its specialisations, will give the same value on all of these embedded graphs:
\begin{corollary}
The polynomial $Q(G\otimes_{\varphi} H; (\alpha, \beta, \gamma), t  )  $ depends only upon $G$, $H$, and the edge $e$ of $H$, and not the particular identification maps $\varphi$ used in the construction of the tensor product.
\end{corollary}
\begin{proof}
In Theorem~\ref{t.tensor}, $Q(G)$, $\kappa$, $\lambda$ and $\mu$, and therefore $Q(G\otimes_{\varphi} H)$, do not use a particular choice of  $\varphi$ in their calculations.   Since this holds for all $t\in\mathbb{C}\backslash \{0,1,-2 \}$, it follows that the polynomials do not depend upon these choices. 
\end{proof}

Theorem~\ref{t.tensor} can also be used to recover a special case of the tensor product formula for the ribbon graph polynomial from \cite{HM}:
\begin{corollary}[\cite{HM}]\label{c.hm}
Let $G$ be an embedded graph, and $H$ be a connected orientable embedded  graph with a distinguished non-loop, non-bridge  edge $e$. Furthermore, let $\varphi$ be a family of identifying maps.
Then when $\kappa, \lambda \neq 0  $, $xy\neq 1$ and $\sqrt{xy}\notin \{0,1,-2\}$, 
\[ (\sqrt{y/x})^{r(G\otimes_{\varphi} H)}R(G\otimes_{\varphi} H; x+1, y, 1/\sqrt{xy})
= \kappa^{r(G)} \lambda^{n(G)} R\Big(G; \frac{\sqrt{xy}\,\lambda}{\kappa}+1,  \frac{\sqrt{xy}\,\kappa}{\lambda}, \frac{1}{ \sqrt{xy}}\Big)  \]
where
\[\kappa=  \frac{(\sqrt{y/x})^{r(H/e)}}{1-xy}  (R(H-e) - xR(H/e)),\quad \quad \lambda=\frac{(\sqrt{y/x})^{r(H-e)}}{1-xy}   (R(H/e) - xR(H-e)),\]
 and where $R(H/e)= R(H/e; x+1, y, 1/\sqrt{xy})$ and   $R(H-e)=R(H-e; x+1, y, 1/\sqrt{xy})$ in all of the expressions.
\end{corollary}
\begin{proof}
 Theorem~\ref{t.tensor} gives
\[Q(G\otimes_{\varphi} H; (b,1,0),t  )  =  Q(G; (\kappa, \lambda, \mu), t  ) = \lambda^{e(G)} Q(G; (\kappa/\lambda, 1, 0), t  ),\]
where $\kappa$, $\lambda$, and $\mu$ are the solutions to Equations~\eqref{t.tensor.e1}-\eqref{t.tensor.e3} with $ (\alpha, \beta, \gamma)=(b,1,0)$. However, as $H$ is orientable, the only way for one cycle to cross another is for there to be a crossing transition at some vertex.  Since crossing transitions have weight zero, it follows that $\mu=Q_{\times}(H-e; (b,1,0), t) =0$, and so $\kappa$ and $\lambda$ can be determined using only \eqref{t.tensor.e1} and \eqref{t.tensor.e2}, giving   $\kappa$ and $\lambda$ in terms of $Q(H-e)$ and $Q(H/e)$.
 The result then follows by using Equation~\eqref{e.qmbr0} to express the transition polynomial expressions in terms of $R(G)$.
\end{proof}

Initially, one may expect to be able to adapt the proof of Corollary~\ref{c.hm} to obtain a tensor product formula for the Penrose polynomial, but a little thought shows that this cannot be done. The key fact that enabled us to obtain Corollary~\ref{c.hm} from Theorem~\ref{t.tensor} was that setting $\gamma=0$ in $Q(G\otimes_{\varphi} H; (\alpha, \beta, \gamma), t  )$,  gave that $\mu=0$ in  $Q(G; (\kappa, \lambda, \mu), t  )$. However, recalling that the weight system for the Penrose polynomial is $(1,0,-1)$, if we set $\beta=0$ in  $Q(G\otimes_{\varphi} H; (\alpha, \beta, \gamma), t  )$, then it does not necessarily follow that $\lambda=0$ in $Q(G; (\kappa, \lambda, \mu), t  )$. (For example, take $H$ to be the plane $\theta$-graph.) Rather than being a hindrance, this is  fortunate in that  it provides a family of identities that relate the Penrose and ribbon graph polynomials. To see this,  Theorem~\ref{t.tensor} and Equation~\eqref{e.qmbr2} give
\begin{equation}\label{e.tp2}
P(G\otimes_{\varphi} H; \lambda  )   =  Q(G; (\kappa, \lambda, \mu), \lambda  ),
\end{equation}
where
\begin{equation}\label{e.tp3}
\begin{split}
t\kappa+t^2\lambda+t\mu&= P( H-e;  \lambda  ) \\
t^2\kappa+t\lambda+t\mu&= P( H/e;   \lambda  ) \\
t\kappa+t\lambda+t^2\mu&= P( H^{\tau(e)}/e;   \lambda  ).
\end{split}
\end{equation}
If it happens that the solution to \eqref{e.tp3} gives $\mu=0$ (which does indeed occur for particular choices of $H$, such as the plane $\theta$-graph ) then \eqref{e.tp2} gives that
\begin{equation}\label{e.tp4}
P(G\otimes_{\varphi} H; t  )
  =   t^{c(G)} \kappa^{r(G)}\lambda^{n(G)}  R(G; t \lambda/\kappa+1,  t \kappa/\lambda, 1/ t),
\end{equation}
again connecting the ribbon graph and Penrose polynomials. Thus, Equation~\eqref{e.dbl2} arises from \eqref{e.tp4} by taking $H$ to be the plane $\theta$-graph.  This family of relations between polynomials can then be used to translate between interpretations of, and results about, the two polynomials. It would be interesting to have a characterisation of the family of 1-variable specialisations of $R(G)$ that appear on the right-hand side of Equation~\eqref{e.tp4} and hence can be written as the Penrose polynomial of some graph.

\medskip

We conclude with a discussion of the relationship between the work presented here and a result of Jaeger.   In \cite{Ja88},  Jaeger found the following connection between the Tutte and transition polynomials, which we have translated to our notation (the polynomials here and in \cite{Ja88} differ by a prefactor and the roles of $\alpha$ and $\beta$ are reversed).
\begin{theorem}[Jaeger \cite{Ja88}]\label{t.ja}
Let $G$ be a connected plane graph, and let $t\in \mathbb{C}\backslash \{0,1,-2\}$ and $\lambda, \mu \in \mathbb{C}\backslash \{0\}$. Then
\[    \tau^{-1} \lambda^{v(G)-1}   \mu^{f(G)-1}   Q(G; (\beta, \alpha, \gamma), t  ) = T(G; (\lambda t)/\mu +1, (\mu t)/\lambda+1)) ,  \]
where $\alpha$, $\beta$, and  $\gamma$ satisfy
\begin{align*}
 t \alpha+ \beta+ \gamma&= t/\mu  +1/\lambda\\
  \alpha+t\beta+ \gamma&= 1/\mu+t/\lambda \\
  \alpha + \beta+t \gamma&= 1/\mu+1/\lambda.
\end{align*}
\end{theorem}

Jaeger's theorem can be obtained as a special case of Theorem~\ref{t.tensor} as follows.
\begin{equation}\label{jder}
        \begin{aligned}
R(G; (\lambda t)/\mu+1, (\mu t)/\lambda , 1/t) &=    (\mu/(\lambda t))^{c(G)}    (\lambda/ \mu)^{v(G)}  \mu^{e(G)}  Q\left(G ;(\mu/\lambda , 1 , 0), t \right)    ) \\
&=    (\mu/(\lambda t))^{c(G)}    (\lambda/ \mu)^{v(G)}  \mu^{e(G)}  Q\left(G ;(1/\lambda , 1/ \mu, 0), \tau \right)) \\
&=    (\mu/(\lambda t))^{c(G)}    (\lambda/ \mu)^{v(G)}  \mu^{e(G)}  Q\left(G ;(\beta, \alpha, \gamma), \tau \right)),
        \end{aligned}
        \end{equation}
where $\alpha$, $\beta$, and  $\gamma$ are the solutions the system of equations given in Theorem~\ref{t.ja}.
In \eqref{jder}, the first equality is by Equation~\eqref{e.qmbr0}, and the third equality follows by an application of Theorem~\ref{t.tensor} where
$H=C_2$ and the tensor product is chosen so that  $G\otimes_{\varphi} C_2 =G$.
Jaeger's Theorem~\ref{t.ja} can then be recovered by restricting to connected plane graphs (so $R(G; (\lambda t)/\mu+1, (\mu t)/\lambda , 1/t)  =  T(G; (\lambda t)/\mu +1, (\mu t)/\lambda+1)$), and by using Euler's formula to rewrite the prefactor.

It is interesting to note that Jaeger  proves his result by, essentially, proving and applying the tensor product for the special case $G\otimes_{\varphi} C_2$ as above. However, the application of Theorem~\ref{t.tensor} is unnecessary  since $(\beta, \alpha, \gamma) = (1/\lambda , 1/\mu, 0)$ is the solution to the system of equations in Theorem~\ref{jder}. Thus the first two lines of  \eqref{jder} give rise to a much simpler proof of Theorem~\ref{t.ja} than the one in \cite{Ja88}. In particular, we see that Theorem~\ref{t.ja} is equivalent to Equation~\eqref{e.qmbr0}, and so it provides no additional information than the simpler formula.

\section*{Acknowledgement}
We thank the anonymous referee for an number of insightful comments and helpful suggestions.

\label{lastpage}

\end{document}